\documentclass[12pt, a4paper]{amsart}
\usepackage{etex}
\usepackage[active]{srcltx}
\usepackage{amsmath,amssymb,amsthm,amsfonts,amsxtra}
\usepackage{enumerate,mathrsfs,comment}
\usepackage[left=25mm,top=25mm,right=25mm,bottom=25mm]{geometry}

\usepackage{makecell} 

\usepackage{multicol}
\usepackage{setspace}
\usepackage{hyperref}
\usepackage{color}
\definecolor{lgray}{gray}{0.9}
\usepackage{tikz}

\usepackage[all,cmtip]{xy}

\newcommand{\CM}{Cohen-Macaulay}

\newcommand{\n}{\mathfrak{n} }
\newcommand{\m}{\mathfrak{m} }
\newcommand{\M}{\mathfrak{M} }

\newcommand{\p}{\mathfrak{p} }

\providecommand{\e}{{\mathcal E}}
\providecommand{\D}{{\mathcal D}}

\newcommand{\Q}{\mathbb{Q} }
\newcommand{\Z}{\mathbb{Z} }
\newcommand{\N}{\mathbb{N} }

\newcommand{\FF}{\mathcal{F}}

\newcommand{\TT}{\mathcal{T} }

\newcommand{\rt}{\rightarrow}
\newcommand{\lrt}{\longrightarrow}

\newcommand{\ee}{\operatorname{e}}

\newcommand{\Id}{\operatorname{Id}}

\newcommand{\Ass}{\operatorname{Ass}}

\newcommand{\chars}{\operatorname{char}}
\newcommand{\depth}{\operatorname{depth}}

\newcommand{\gr}{\operatorname{gr}}

\newcommand{\hgt}{\operatorname{height}}

\newcommand{\Mod}{\operatorname{Mod}}

\newcommand{\Supp}{\operatorname{Supp}}
\newcommand{\supp}{\operatorname{supp}}
\newcommand{\Spec}{\operatorname{Spec}}

\newcommand{\injdim}{\operatorname{injdim}}

\newcommand{\rank}{\operatorname{rank}}

\newcommand{\Ext}{\operatorname{Ext}}

\newcommand{\Tot}{\operatorname{Tot}}

\theoremstyle{plain}

\newtheorem{theorem}{Theorem}[section]
\newtheorem{corollary}[theorem]{Corollary}
\newtheorem{lemma}[theorem]{Lemma}
\newtheorem{proposition}[theorem]{Proposition}

\theoremstyle{definition}
\newtheorem{definition}[theorem]{Definition}
\newtheorem{s}[theorem]{}
\newtheorem{remark}[theorem]{Remark}
\newtheorem{example}[theorem]{Example}

\newtheorem*{question*}{\it Question}
\newtheorem{remarks}[theorem]{Remarks}
\newtheorem*{example*}{Example}
\newtheorem*{claim*}{\it Claim}
\newtheorem*{note*}{\it Note}

\newtheorem{claim}{\bf Claim}

\newtheorem*{subclaim*}{\it Subclaim}

\title[Graded components of local cohomology modules]{Graded components of local cohomology modules of $\mathfrak{C}$-monomial ideals in characteristic zero}
\date{\today}


\author{Tony J. Puthenpurakal}
\address{Department of Mathematics, Indian Institute of Technology Bombay, Powai, Mumbai 400076, India}
\email{\href{mailto:tputhen@math.iitb.ac.in}{tputhen@math.iitb.ac.in}}
\author{Sudeshna Roy}
\address{School of Mathematics, Tata Institute of Fundamental Research, Homi Bhabha Road, Mumbai 400005, India}
\email{\href{mailto:sudeshnaroy.11@gmail.com}{sudeshnaroy.11@gmail.com}}
\begin{document}

\begin{abstract}
Let $A$ be a commutative Noetherian ring of characteristic zero and $R=A[X_1, \ldots, X_d]$ be a polynomial ring over $A$ with the standard $\N^d$-grading. Let $I\subseteq R$ be an ideal which can be generated by elements of the form $aU$ where $a \in A$ (possibly nonunit) and $U$ is a monomial in $X_i$'s. We call such an ideal as a `$\mathfrak{C}$-monomial ideal'. Local cohomology modules supported on monomial ideals gain a great deal of interest due to their applications in the context of toric varieties. It was observed that for $\underline{u} \in \Z^d$, their $\underline{u}^{th}$ components depend only on which coordinates of $\underline{u}$ are negative. In this article, we show that this statement holds true in our general setting, even for certain invariants of the components. We mainly focus on the Bass numbers, injective dimensions, dimensions, associated primes, Bernstein-type dimensions, and multiplicities of the components. Under the extra assumption that $A$ is regular, we describe the finiteness of Bass numbers of each component and bound its injective dimension by the dimension of its support. 	Finally, we present a structure theorem for the components when $A$ is the ring of formal power series in one variable over a characteristic zero field. 
\end{abstract}
\maketitle

\section{Introduction}	

Let $K$ be a field and $S=K[X_1, \ldots, X_d]$ be a polynomial ring
over $K$. In recent years 
explicit information about the structure of local cohomology modules $H^j_J(S)$ supported on monomial ideals $J$ has been obtained based on the following two facts:
\begin{enumerate}[\rm (i)]
\item they have a structure of $\Z^d$-graded modules which allows one to study these modules using combinatorial approach, see \cite{Mus}, \cite{Ter}, \cite{Yan2001}.
	
\item they have a structure of $\D$-modules, see \cite{Alv}, \cite{AlvZar}, \cite{AlvGarZar}.
\end{enumerate}
%
In fact, there is a dictionary between the $\Z^d$-graded structure and the $\D$-module structure of $H^j_J(S)$, see \cite[III.2.9]{JAM_LC}.
These modules are important for their connection
to the cohomology of coherent sheaves on a 
toric variety over a field (see \cite{EisMusSti}) and useful for constructing examples and non-examples. 
However, it is long known that every local cohomology is associated to certain sheaf cohomology, see \cite[Theorem 12.41]{24Hrs}. The goal of this article
is to analyze the $\Z^d$-graded structure of local cohomology modules in the following framework:

\s {\it Setup}.\label{sa_intro} Let $A$ be any commutative Noetherian ring containing a field of characteristic zero and $R=A[X_1, \ldots, X_d]$ be a polynomial ring over $A$. Consider an $\N^d$-grading on $R$ by setting $\deg z=0$ for all $z \in A$ and $\deg X_i=e_i \in \N^d$ for $i=1, \ldots, d$. Let $I \subseteq R$ be a $\mathfrak{C}$-monomial ideal, that is, $I$ can be generated by elements of the form $aU$ where $a \in A$ (possibly nonunits) and $U$ is a monomial in $X_1, \ldots, X_d$.

It is well-known that the local cohomology modules are rarely finitely generated. The prime way to tackle local cohomology modules supported on monomial ideals is to restore finiteness by considering their components under suitable grading. In our setting, 
although $H^i_I(R)$ has an induced $\Z^d$-graded structure for each $i \geq 0$, 
its components are not necessarily finitely generated, see Theorem \ref{non-fg}. So 
we assume that the base ring has characteristic zero and use the techniques of D-modules, which was pioneered by G. Lyubeznik in \cite{Lyu93}. For this study, we introduce the notion of $\Z^d$-graded generalized Eulerian $A_{\underline{1}}(A)$-module (Definition \ref{MGE}). This is motivated from the authors' previous analysis in \cite{Put2015}, \cite{Put}, \cite{PutRoy22}. 
Previously, researchers realized that to understand the $\Z^d$-graded structure of $H^j_J(S)$ 
one only needs to study their $\underline{a}$-th pieces 
for $\underline{a} \in \{-1,0\}^d$. This observation is reflected in our results regarding Bass numbers (Theorem \ref{multi-bass}), injective dimensions (Theorem \ref{multi-injdim-dim}) and associated primes (Theorem \ref{multi-ass}) of components of $H^i_I(R)$. Moreover, under the extra assumption that $A$ is a regular ring, in Theorem \ref{injdim-dim} we prove 
that 
%
%
\begin{enumerate}[\rm (i)]
\item all Bass numbers of $H^i_I(R)_{\underline{u}}$ are finite for every $\underline{u} \in \Z^d$.

\item $\injdim H^i_I(R)_{\underline{u}} \leq \dim \supp H^i_I(R)_{\underline{u}}$ for fix $\underline{u} \in \Z^d$.
\end{enumerate}
We further show that if $I$ is a usual monomial ideal, then $H^i_I(R)_{\underline{u}}$ is a free $A$-module. We call an ideal a \emph{usual monomial ideal} if it 
is generated by 
monomials in $X_1, \ldots, X_d$. 

Let $K$ be a field of characteristic zero. 
For the case when $A=K[[Y]]$ we establish a structure theorem (Theorem \ref{struc}) for the components of $H^i_I(R)$ illustrating their torsion parts and torsion-free parts. It appears that their behavior is almost identical to the mixed characteristic case which was studied by the authors in \cite{PutRoy22_pre}.  

We now assume that $A=K[[Y_1, \ldots, Y_m]]$.
When $m=0$, it follows from \cite[Theorem 2·11]{Yan2001} that the multiplicity of the $\underline{u}^{th}$-component of $H^i_I(R)$ depends only on which coordinates of $\underline{u}$ are negative. This motivates us to investigate the asymptotic behavior of the multiplicities of the components in the generality.
Let $\D$ be the ring of $A$-linear differential operators on $R$ and let $L$ be a finitely generated $\Z^d$-graded $\D$-module. We define a $\Z^d$-graded Bernstein type good filtration $\FF$ on $\D$ (see \ref{D-fil}) and a $\Z^d$-graded $\FF$-compatible good filtration $\Gamma$ on $L$ (see \ref{comp_fil}). We show that for each $\underline{u}\in \Z^d$ there is an induced filtration $\Gamma_{\underline{u}, *}$ on $M_{\underline{u}}$ such that the associated graded module $\gr_{\Gamma_{\underline{u}, *}}(M_{\underline{u}})$ is a finitely generated $\gr_{\FF_{\underline{0}, *}}(\D_{\underline{0}})$-module. We observe that $\gr_{\FF_{\underline{0}, *}}(\D_{\underline{0}})$ is a polynomial ring over $A$ in $(m+d)$-variables, see \eqref{polynomialRing}. It allows us to define Bernstein type dimension and multiplicity of $M_{\underline{u}}$ (see \ref{dim-mul}) as in \cite[Definition 6.3, Chapter 2]{Bjo}.
Similar to other invariants, we prove that the Bernstein type dimension and the multiplicity of $H^i_I(R)_{\underline{u}}$, for $\underline{u} \in \Z^d$, depend only on which coordinates of $\underline{u}$ are negative (Theorem \ref{multi-Bdim-mul}).

This article is structured as follows. In Section 2, we introduce and discuss certain properties of multi-graded generalized Eulerian modules. In Section 3, we establish their rigidity property. We connect this class to the class of straight modules due to K. Yanagawa in Section 4. Section 5 is devoted to study some finiteness properties and asymptotic behaviour of specific invariants of the components of $H^i_I(R)$ 
under the hypotheses in \ref{sa_intro}. Suppose that $\D$ denotes the ring of $A$-linear differential operators on $R$. Notice there is a natural induced $\Z^d$-grading on $\D$. In Section 6, we develop a Bernstein-type theory for the components of finitely generated $\Z^d$-graded $\D$-modules 
when $A=K[[Y_1, \ldots, Y_m]]$ is a power series ring over a field $K$ of characteristic zero. In Section 7, we produce a structure theorem for the components of $H^i_I(R)$ when $A=K[[Y]]$. Finally, in Section 8, we give examples to show that both torsion parts and torsion-free parts of the components can be non-finitely generated.

Throughout this article, we use $\mathcal{S}$ to denote $[d]:=\{1, \ldots, d\}$.

\section{Multi-graded generalized Eulerian} 

\s {\bf Setup}.\label{sa} Let $A$ be a commutative Noetherian ring containing a field of characteristic zero. Let $R=A[X_1, \ldots, X_d]$ be a polynomial ring over $A$. Consider an $\N^d$-grading on $R$ by setting $\deg z=0$ for all $z \in A$ and $\deg X_i=e_i$ for $i=1, \ldots, d$. Let $I=(a_1U_1, \ldots, a_tU_t)$ be a $\mathfrak{C}$-monomial ideal with $a_i \in A$ and $U_i$'s are monomials in $X_1, \ldots, X_d$. Since $I$ is a multi-homogeneous ideal, $H^i_I(R)$ has an induced $\Z^d$-graded structure. We denote $\underline{u}:=(u_1, \ldots, u_d) \in \Z^d$
and set $M:=H^i_I(R)=\bigoplus_{\underline{u}\in \Z^d} M_{\underline{u}}$.

\vspace{0.2cm}
Let $A_{\underline{1}}(A)=A\langle X_1, \ldots, X_d, \partial_1, \ldots,\partial_d\rangle$ denote the Weyl algebra over $A$, where $\partial_i=\partial/\partial X_i$ for $i=1, \ldots, d$. It is naturally $\Z^d$-graded with $\deg z=0$ for all $z \in A$, $\deg X_i=e_i$ and $\partial_i=-e_i$ for $i=1, \ldots, d$.

For $i=1, \ldots, d$,
\[\e_i:=X_i\partial_i\]
is known as the {\it $i^{th}$ Euler operator}, see \cite{24Hrs}. We say $M$ is a $\Z^d$-{\it graded Eulerian} $A_{\underline{1}}(A)$-module if for each multihomogeneous element $y$ in $M$ with $\deg y=\underline{u}$, 
\[\e_i \cdot y= u_i \cdot y \quad \mbox{ for all } 1 \leq i \leq d.\] 

\begin{example}\label{poly-R-E}
Let $A=K$ and $E:=E_R(K)$ denote the ${}^*$injective hull of $K$ over $R$. Then it is well-known that
\[E(\underline{1})\cong H^d_{(X_1, \ldots, X_d)}(R)\cong 
X_1^{-1}\cdots X_d^{-1}K[X_1^{-1}, \ldots, X_d^{-1}].\]
Notice $E(\underline{1})$ has natural grading induced from the ring $R$. Fix $i$. For any $\underline{u} \in \Z^d$, we have 
\begin{align*}
\e_i \cdot \underline{X}^{\underline{u}}=  X_1^{u_1} \cdots X_{i-1}^{u_{i-1}} X_{i+1}^{u_{i+1}} \cdots X_d^{u_d} \left(\e_i \cdot X_i^{u_i}\right)= u_i \underline{X}^{\underline{u}}.
\end{align*}
So both $R$ and $E(\underline{1})$ are $\Z^d$-graded Eulerian $A_{\underline{1}}(K)$-module. 
\end{example}

\begin{definition}\label{MGE}
We say a $\Z^d$-graded $A_{\underline{1}}(A)$-module $M=\bigoplus_{\underline{u}\in \Z^d}M_{\underline{u}}$ is \emph{generalized Eulerian} if for each multihomogeneous element $y$ of $M$ with $\deg y = \underline{u}$, there exists a positive integer $a$ (possibly depending on $y$) such that
\[\left(\e_i-u_i\right)^a \cdot y=0\ \quad \mbox{for } i=1, \ldots, d.\]
\end{definition}


The 
next result says that 
the class of multigraded generalized Eulerian modules is closed under extension, as it was observed for the class of graded generalized Eulerian modules.
\begin{proposition}\label{exact}
	Let $0 \rt M_1 \xrightarrow{\alpha_1} M_2 \xrightarrow{\alpha_2} M_3 \rt 0$ be a short exact sequence of $\Z^d$-graded $A_{\underline{1}}(A)$-modules {\rm(}all maps are multihomogeneous{\rm)}. Then the following are equivalent:
	\begin{enumerate}[\rm (1)]
		\item $M_2$ is $\Z^d$-graded generalized Eulerian.
		\item $M_1$ and $M_3$ are $\Z^d$-graded generalized Eulerian.
	\end{enumerate}
\end{proposition}

\begin{proof}
A proof of this is, with straightforward modifications, the proof given in \cite[Proposition 2.1]{Put2015}. 
\end{proof}

From Definition \ref{MGE}, we perceive that a multigraded  module $M$ is generalized Eulerian if given any $y \in M_{\underline{u}}$ each coordinate 
$u_i$ satisfies certain equation, that is, we need to execute coordinate wise study.
Therefore, 
applying the same techniques, used in \cite{Put} and \cite{PutSin} to prove results in the graded case,
one can get analogues results in the multigraded setting. 
Particularly, the following crucial results can be obtained. To illustrate our claim, we give a proof of one result.

\begin{proposition}\label{genEur-XY}
Let $M$ be a $\Z^d$-graded generalized Eulerian $A_{\underline{1}}(A)$-module. Fix $i$. Then $H_j(X_i; M)$ is a $\Z^d$-graded generalized Eulerain $A_{\underline{1}-e_i}(A)$-module where $j=0,1$. Moreover, $H_j(X_i; M)_{\underline{u}}$ $ \neq 0$ only if $u_i=0$. 
\end{proposition}

\begin{proof}
The map $M(-e_i) \overset{\cdot X_i}{\lrt} M$ is $\Z^d$-graded $A_{\underline{1}-e_i}$-linear. So $H_j(X_i; M)$ is a $\Z^d$-graded $A_{\underline{1}-e_i}(A)$-module for $j=0, 1$. Thus we have an exact sequence of $\Z^d$-graded $A_{\underline{1}-e_i}$-modules \[0 \rt H_1(X_i; M) \rt M(-e_i) \overset{\cdot X_i}{\lrt} M \rt H_0(X_i; M)\rt 0.\]
Let $\eta \in H_1(X_i; M)(e_i) \subseteq M$ be a nonzero multihomogeneous element with $\deg (\eta)=\underline{u} \in \Z^d$. Note that $H_1(X_i; M)(e_i)_{\underline{u}}=H_1(X_i; M)_{\underline{u}+e_i}$. Since $M$ is a $\Z^d$-graded generalized Eulerian module, there exists $b \geq 1$ such that
\[(X_t \partial_t-u_t)^b \eta=0.\]
for all $t=1, \ldots, d$ (particularly, for $t \neq i$). Hence $H_1(X_i; M)$ is a $\Z^d$-graded generalized Eulerian $A_{\underline{1}-e_i}$-module.

As $\partial_iX_i-X_i \partial_i= 1$ so we can write \[0= (\partial_iX_i-(u_i+1))^b\eta= (*)X_i\eta + (-1)^b(u_i+1)^b\eta.\] 
Besides, $X_i \eta=0$ and $\eta \neq 0$. Thus we get $u_i=-1$, i.e., 
	\begin{align*}
	\eta \in H_1(X_i; M)(e_i)_{(u_1, \ldots,u_{i-1},-1, u_{i+1}, \ldots, u_d)}= H_1(X_i; M)_{(u_1, \ldots,u_{i-1},0, u_{i+1}, \ldots, u_d)}.
	\end{align*} 

\vspace{0.2cm}	
Next, let $\eta' \in H_0(X_i; M)$ be non-zero and multihomogeneous of multidegree $\underline{v}$. Therefore $\eta' \in (M/X_iM)_{\underline{v}}$ and hence $\eta'= \beta+X_iM$ for some $\beta \in M_{\underline{v}}$. Since $M$ is multigraded generalized Eulerian, there exists $a \geq 1$ such that 
\[(X_t \partial_t-v_t)^a \beta=0\] 
for all $t=1, \ldots, d$, (particularly, for all $t \neq i$). The above relation also holds in $M/X_iM$. It thus follows that $H_0(X_i; M)$ is a $\Z^d$-graded generalized Eulerian $A_{\underline{1}-e_i}$-module.
	
	Further observe that $(X_i \partial_i-v_i)^a= X_i \cdot (*)+(-1)^av_i^a$. Thus \[X_i \cdot (*)\beta+(-1)^av_i^a\beta=0.\] In $M/X_iM$, we have $(-1)^av_i^a\eta'=0$. Now $\eta' \neq 0$ implies that $v_i=0$, i.e., \[\eta' \in H_0(X_1; M)_{(v_1, \ldots,v_{i-1},0, v_{i+1}, \ldots, v_d)}.\] 
	The result follows.
\end{proof}

Similarly, we can prove the following. 
\begin{proposition}\label{genEur-partial}
Let $M$ be a $\Z^d$-graded generalized Eulerian $A_{\underline{1}}(A)$-module. Fix $i$. Then $H_j(\partial_i; M)(-e_i)$ is a $\Z^d$-graded generalized Eulerain $A_{\underline{1}-e_i}(A)$-module where $j=0,1$. Moreover, $H_j(\partial_i; M)_{\underline{u}}$ $ \neq 0$ only if $u_i=-1$ where $j=0,1$. 
\end{proposition}

The subsequent result is also essential. 
	
\begin{theorem}[with the hypotheses as in \ref{sa}] Then $H^i_I(R)$ is a $\Z^d$-graded generalized Eulerian $A_{\underline{1}}(A)$-module for all $i \geq 0$.
\end{theorem} 

 \begin{proof}
Since $R$ is a $\Z^d$-graded $A_{\underline{1}}(A)$-module and $I \subseteq R$ is a multihomogeneous ideal, from the corresponding {\v C}ech complex, it is easily seen that $H^i_I(R)$ has an induced $\Z^d$-grading. For a precise argument one can refer to \cite[Step-1, Proof of Lemma 3.8]{Put}.

Take $y \in H^i_I(R)_{\underline{u}}$.
Fix $j$. Set $R^{(j)}=A[X_1, \ldots, X_{j-1}, X_{j+1}, \ldots, X_d]$ and consider $R$ as graded by setting $\deg z=0$ for all $z \in R^{(j)}$ and $\deg X_j=1$. In view of multi-degrees assigned on the variables $X_j$’s, one can easily verify that any multihomogeneous ideal $I$ in $R$ with respect to its $\N^d$-grading (as in \ref{sa}) is homogeneous with respect to the grading on $R$ defined above. Hence $H^i_I(R)$ is a graded module over both $R$ and $A_1\left(R^{(j)}\right)$-module. Moreover, for each $u \in \Z$, 
\[H^i_I(R)_u=\bigoplus_{\substack{\underline{v} \in \Z^d, v_i=u}} H^i_I(R)_{\underline{v}}.\] 
Notice that $\deg y=u_j$ in the current setting. By \cite[Theorem 3.7]{Put}, $H^i_I(R)$ is a graded generalized Eulerian $A_1\left(R^{(j)}\right)$-module. So there exists some $a_j>0$ such that
\[(\e_j-\deg y)^{a_j} \cdot y=0,\]
that is, $(\e_j-u_j)^{a_j}\cdot y=0$. This holds true for all $1 \leq j \leq d$. Hence $H^i_I(R)$ is a $\Z^d$-graded generalized Eulerian $A_{\underline{1}}(A)$-module.
 \end{proof} 

We need the following notions 
for the succeeding result. 

We say $Y$ is a \emph{multihomogeneous closed subset} of $\Spec(R)$ if $Y = V(f_1, \ldots, f_s)$, where $f_i$'s are multihomogeneous polynomials in $R$. We say $Y$ is a \emph{multihomogeneous locally closed subset} of $\Spec(R)$ if $Y = Y_1 - Y_2$, where $Y_1, Y_2$ are multihomogeneous closed subsets of $\Spec(R)$.

Recall `$\Lambda$' stated in \cite[3.1]{Put}. Since its description is a bit involved, we decided to skip. 

\begin{corollary}\label{genEul-LyuFun}
Let ${}^*\Mod(R)$ denote the category of $\Z^d$-graded $R$-modules and $\mathcal{T}$ be a $\Z^d$-graded Lyubeznik functor on ${}^*\Mod(R)$. 
Then $\mathcal{T}(R)$ is a $\Z^d$-graded generalized Eulerian $A_{\underline{1}}(\Lambda)$-module.
\end{corollary}

\begin{proof}
The statement follows from straightforward modifications of the proofs of \cite[Step-2 in Lemma 3.8, Theorem 3.7]{Put} in the multigraded setting. We need to work with multihomogeneous locally closed subset of $\Spec(R)$. 
\end{proof}

For $\underline{u}\in \Z^d$, we set $|\underline{u}|=\sum_{i=1}^d u_i$. Associated to a $\Z^d$-graded $A_{\underline{1}}(A)$-module $M$, there is a graded $A_{d}(A)$-module $\Tot(M)=\bigoplus_{u \in \Z} \Tot(M)_u$, where for each $u \in \Z$
\[\Tot(M)_u=\bigoplus_{\underline{u}\in \Z^d, ~|\underline{u}|=u}M_{\underline{u}}.\]

\begin{lemma}\label{Rel-multigrd-grd-genEur}
Let $M$ be a $\Z^d$graded generalized Eulerian $A_{\underline{1}}(A)$-module. Then $\Tot(M)$ is a graded generalized Eulerian $A_{d}(A)$-module.
\end{lemma}

\begin{proof}
Observe that $\e_i$ and $\e_j$ commute with each other for any pair $i,j$. Take $y \in M_{\underline{u}}$. Since $M$ is $\Z^d$graded generalized Eulerian, there exists some $a>0$ such that $(\e_i-u_i)^{a}=0$ for $i=1, \ldots, d$. So 
\begin{align*}
\left(\sum_{i=1}^d \e_i - \deg(y)\right)^{da} \cdot y &= \left(\sum_{i=1}^d \e_i - \sum_{i=1}^d u_i\right)^{da} \cdot y\\
&=\big((X_1-u_1)+\cdots+(X_d-u_d)\big)^{da} \cdot y\\
&=\left(\sum_{i_1+\cdots+i_d=da} (X_1-u_1)^{i_1} \cdots (X_d-u_d)^{i_d}\right) \cdot y\\
&=0, 
\end{align*} 
as $i_j \geq a$ for some $j$. Hence $\Tot(M)$ is  graded generalized Eulerian.
\end{proof}

\begin{definition}\label{holo-multi}
%
Let $A=K[[Y_1, \ldots, Y_m]]$ be the ring of formal power series in $m$ variables over a field $K$ of characteristic zero. Let $\Lambda= A\langle \delta_1, \ldots, \delta_m\rangle$ be the ring of $K$-linear differential operators on $A$. In light of \cite[4.1]{Put},
we say a finitely generated $\Z^d$-graded left $A_{\underline{1}}(\Lambda)$-module $M$ is {\it holonomic} if it is zero, or if the $A_d(\Lambda)$-module $\Tot(M)$ has 
dimension $m+d$.
\end{definition}

\begin{lemma}\label{fl_hol}
Every $\Z^d$-graded holonomic $A_{\underline{1}}(\Lambda)$-module has finite length.
\end{lemma}

\begin{proof}
Let $M$ be a holonomic $A_{\underline{1}}(\Lambda)$-module. Set $L:=\Tot(M)$. From the definition we get that $L$ is a holonomic $A_d(\Lambda)$-module, see \cite[4.1]{Put}. Let $\widehat{R}=K[[Y_1, \ldots, Y_m, X_1, \ldots, X_d]]$ and $\mathcal{M}$ denote the unique homogeneous maximal ideal of $R$. Then one can verify that $\widehat{R}=\widehat{R_{\mathcal{M}}}$, the completion of $R_{\mathcal{M}}$ with respect to $\mathcal{M}$. We denote the ring of $K$-linear differential operators on $\widehat{R}$ by $\mathcal{A}:=\widehat{R}\langle \delta_1, \ldots, \delta_m, \partial_1, \ldots, \partial_d\rangle$, where $\delta_i=\partial/\partial Y_i$ and $\partial_j=\partial/ \partial X_j$ for $i=1, \ldots, m$ and $j=1, \ldots, d$. We put $L'=\widehat{R} \otimes_R L$. Then by \cite[Theorem 4.20]{Put}, $L'$ is a holonomic $\mathcal{A}$-module. Thus $L'$ has finite length. Since $L$ is a finitely generated $A_d(\Lambda)$-module, $L'$ is a finitely generated $\mathcal{A}$-module. Consequently, $L'$ is a Noetherian module, as $\mathcal{A}$ is a Noetherian ring. Suppose that $L$ is not Artinian. Then we have a strictly decreasing chain of graded $A_d(\Lambda)$-submodules of $L$:
\[L=L_1\supsetneq L_2 \supsetneq \cdots \supsetneq L_i \supsetneq L_{i+1} \supsetneq \cdots.\]
As the composition map $R\to R_{\mathcal{M}} \to \widehat{R}$ induces a faithfully flat functor from ${}^*\Mod(R)$ to $\Mod(R)$ so for each $i \geq 1$,
\begin{enumerate}
\item $L'_i=\widehat{R} \otimes_R L_i$ is a submodule of $L'$,
\item $L'_i \supsetneq L'_{i+1}$.
\end{enumerate}
This contradicts the fact that $L'$ is a finite length $\mathcal{A}$-module, see \cite[Theorem 7.13, Chapter 2]{Bjo}. It thus follows that $M$ has finite length.
\end{proof}

It is well-known that if $I$ is a monomial ideal in $S=K[X_1, \ldots, X_d]$, then $\dim_K H^i_I(S)_{\underline{u}}< \infty$ for all $\underline{u} \in \Z^d$ (see \cite{Mus}, \cite{Ter}, \cite{Yan2001}). 
We now give an alternative proof of this fact.

\begin{theorem}\label{hol-mon}
	Suppose that $M$ is a holonomic, $\Z^d$-graded generalized Eulerian $A_{\underline{1}}(K)$-module. Then $\dim_K M_{\underline{u}}< \infty$ for all $\underline{u} \in \Z^d$.
\end{theorem}

\begin{proof}
	First, we show the following:
	
	\begin{claim}
		If $\dim_K M_{\underline{a}}< \infty$ for some $\underline{a}\in \Z^d$, then $\dim_K M_{\underline{u}} < \infty$ for all $\underline{u}\in \Z^d$.
	\end{claim}
	
	We use induction on $d$. Let $d=1$. Given assumptions on $M$ implies that $\Tot(M)$ is a holonomic, graded generalized Eulerian $A_1(K)$-module by Definition \ref{holo-multi} and Lemma \ref{Rel-multigrd-grd-genEur}, respectively. From \cite[Lemma 10.1]{Put}, it follows that $\dim_K \Tot(M)_w< \infty$ for all $w \in\Z$. 
	Hence $\dim_K M_{\underline{u}} \leq \dim_K \Tot(M)_{|\underline{u}|}< \infty$ for all $\underline{u} \in \Z^d$.  
	
	Next, suppose that $d>1$ and the result is known for $d-1$. Fix $i \leq d$. Consider the exact sequence of multigraded 
$A_{\underline{1}-e_i}(K)$-modules
	\begin{equation}\label{Kos_1}
	0 \to H_1(X_i; M) \to M(-e_i) \xrightarrow{\cdot X_i} M \to H_0(X_i; M) \to 0.
	\end{equation}
	Notice each element of $A_{\underline{1}-e_i}(K)$ commutes with $X_i$. Since $\Tot(H_j(X_i; M))$ is a holonomic $A_{d-1}(K)$-module by \cite[Theorem 1.6.2]{Bjo}, $H_j(X_i; M)$ is a multigraded holonomic $A_{\underline{1}-e_i}(K)$-module for $j=0,1$. Moreover, $H_j(X_i; M)$ is a multigraded generalized Eulerian $A_{\underline{1}-e_i}(K)$-module by Proposition \ref{genEur-XY}. 
	From \eqref{Kos_1} we get an exact sequence of $K$-vector spaces
	\begin{equation}\label{comp_fd}
	0 \to H_1(X_i; M)_{\underline{u}} \to M_{\underline{u}-e_i} \xrightarrow{\cdot X_i} M_{\underline{u}}.
	\end{equation}
	for all $\underline{u} \in \Z^d$. 
	As $\dim_K M_{\underline{a}} < \infty$
	so $\dim_K H_1(X_i; M)_{\underline{a}+e_i}< \infty$. By induction hypothesis it follows that $\dim_K H_1(X_i; M)_{\underline{u}}< \infty$ for all $\underline{u} \in \Z^d$. For $\underline{u}=\underline{a}$, from \eqref{comp_fd} we get that $\dim_K M_{\underline{a}-e_i}< \infty$. Similarly by taking the following part of \eqref{Kos_1}
	\[ M(-e_i) \xrightarrow{\cdot X_i} M \to H_0(X_i; M) \to 0,\]
	one can show that $\dim_K M_{\underline{a}+e_i}< \infty$. 
Iterating the above steps we get the claim.
	
	\vspace{0.15cm}
	It is now enough to prove that $\dim_K M_{\underline{w}}< \infty$ for some $\underline{w}\in \Z^d$. We show $\dim_K M_{\underline{0}}< \infty$. Set $\D:=A_{\underline{1}}(K)$. For any $a \in K,$
	\[X_i\left(\e_i-a\right)=X_i\left(X_i\partial_i-a\right)=X_i(\partial_iX_i-1-a)=(X_i\partial_i)X_i-(a+1)X_i=\left(\e_i-\overline{a+1}\right)X_i.\] 
	Thus for any $j \geq 2$, 
	\begin{align*}
	X_i^j\partial_i^j&=X_i^{j-1}\e_i\partial_i^{j-1}\\
	&=X_i^{j-2}\left(X_i\e_i\right)\partial_i^{j-1}\\
	&=X_i^{j-2}\left(\e_i-1\right)X_i\partial_i^{j-1}\\	
	&=X_i^{j-2}\left(\e_i-1\right)\e_i\partial_i^{j-2}
	\end{align*}
	
\begin{align*}
	&\vdots\\
	&=\left(\e_i-j+1\right) \cdots \left(\e_i-1\right)\e_i.
	\end{align*}
	Hence $\D_{\underline{0}}=K\left[\e_1, \ldots, \e_d\right]$.
	
	Let $V$ be a $\D_{\underline{0}}$-submodule of $M_{\underline{0}}$. It can be seen that $\D V \cap M_{\underline{0}}=V$. 
	\begin{claim} 
		$M_{\underline{0}}$ is a finitely generated $\D_{\underline{0}}$-module. 	
	\end{claim}
	Let
	\[V_1 \subseteq V_2 \subseteq \cdots \subseteq V_{c-1} \subseteq V_c \subseteq V_{c+1} \subseteq \cdots\]
	be an ascending chain of $\D_{\underline{0}}$-submodules of $M_{\underline{0}}$. This induces an ascending chain of $\D$ submodules of $M$
	\[0 \subseteq \D V_0 \subseteq \D V_1 \subseteq \cdots \subseteq \D V_{c-1} \subseteq \D V_c \subseteq \D V_{c+1} \subseteq \cdots\]
	of $M$. As $\D$ is left Noetherian and $M$ is holonomic so $M$ is Noetherian. Thus there exists a non-negative integer $t$ such that $\D V_j=\D V_t$. Hence $V_j=V_t$ for all $j \geq t$. Therefore, $M_{\underline{0}}$ is a Noetherian $\D_{\underline{0}}$-module. The claim follows. 
	
	Let $\{m_1, \ldots, m_s\}$ be a finite set of generators of $M_{\underline{0}}$ as a $\D_{\underline{0}}$-module. Since $M$ is $\Z^d$-graded generalized Eulerian, 
	\[\e_i^{a_j} \cdot m_j=\left(\e_i-|m_j|\right)^{a_j} \cdot m_j=0 \quad \mbox{ for all } j=1, \ldots, s.\]
	Set $a=\max \{a_1, \ldots, a_s\}$. 
	Let $m=\sum_{j=1}^s c_j m_j \in M_{\underline{0}}$ with $c_j\in \D_{\underline{0}}$. Then for $i=1, \ldots, r$ 
	\[\e_i^a \cdot m =
	\sum_{j=1}^s c_j \cdot \left(\e_i^a \cdot m_j\right)=0,\]
as $\D_{\underline{0}}$ is a commutative ring. Hence $M_{\underline{0}}$ is a finitely generated $\frac{\D_{\underline{0}}}{\left(\e_1^a, \ldots, \e_d^a\right)}=\frac{K\left[\e_1, \ldots, \e_d\right]}{\left(\e_1^a, \ldots, \e_d^a\right)}$-module which implies that $\dim_K M_{\underline{0}}< \infty$.
\end{proof}

\section{Rigidity}

Let $\mathcal{S}$ denote the set $\{1, \ldots, d\}$ and $U$ be a subset (may be empty) of $\mathcal{S}$. By $\underline{a}^U=(a^U_1, \ldots, a^U_d) \in \{0, 1\}^d$ we denote a vector in $\Z^d$ such that $a^U_i=0$ when $i \in U$ and $a^U_i=-1$ when $i \in \mathcal{S} \backslash U$. 

\noindent
\begin{minipage}{.7\textwidth}
\hspace{0.25cm} We define a {\it block} to be
\[\mathcal{B}(\underline{a}^U)=\{\underline{u} \in \Z^d \mid u_i \geq 0 \mbox{ if } i \in U \mbox{ and } u_i \leq -1 \mbox{ if } i \notin U \}.\]

Let $r=2$. Then $\underline{a}^{\phi}=(-1,-1),~ \underline{a}^{\{1\}}=(0,-1),~\underline{a}^{\{2\}}=(-1,0),$ and $\underline{a}^{\{1,2\}}=(0,0)$. 
In the figure on the right, $\mathcal{B}(\underline{a}^{\{1,2\}}), \mathcal{B}(\underline{a}^{\{2\}}), \mathcal{B}(\underline{a}^{\phi}),$ and $\mathcal{B}(\underline{a}^{\{1\}})$ are presented by the shaded parts $B_1, B_2, B_3,$ and $B_4$ respectively.
\end{minipage}
\begin{minipage}{.25\textwidth}
\begin{center}
	\begin{tikzpicture}[scale=0.15]
	\draw[->] (-8.5,0)--(9,0) node[right]{$u$};
	\draw[->] (0,-8.5)--(0,9) node[above]{$v$};
	\draw[densely dotted, red] (-2,8.5)--(-2,-8.5);
	\draw[densely dotted, red] (-8.5,-2)--(8.5, -2);
	\draw[fill=cyan,fill opacity=0.35,draw=none] (0,8.5)--(0,0)--(8.5,0)--(8.5,8.5)--(0,8.5);
	\draw[fill=cyan,fill opacity=0.35,draw=none] (-2,8.5)--(-2,0)--(-8.5,0)--(-8.5,8.5)--(-2,8.5);
	\draw[fill=cyan,fill opacity=0.35,draw=none] (-2,-8.5)--(-2,-2)--(-8.5,-2)--(-8.5,-8.5)--(-2,-8.5);
	\draw[fill=cyan,fill opacity=0.35,draw=none] (0,-8.5)--(0,-2)--(8.5,-2)--(8.5,-8.5)--(0,-8.5);
	\node[blue,draw=none] at (4,-6) {$B_4$};
	\node at (0,-2){\textcolor{red}{$\bullet$}};
	\node[blue,draw=none] at (4,4) {$B_1$};
	\node at (0,0){\textcolor{red}{$\bullet$}};
	\node[blue,draw=none] at (-6,4) {$B_2$};
	\node at (-2,0){\textcolor{red}{$\bullet$}};
	\node[blue,draw=none] at (-6,-6) {$B_3$};
	\node at (-2,-2){\textcolor{red}{$\bullet$}};
	\node[blue,draw=none] at (3,-3) {\tiny $(0,-1)$};
	\node[blue,draw=none] at (2.5,1) {\tiny $(0,0)$};
	\node[blue,draw=none] at (-5,1) {\tiny $(-1,0)$};
	\node[blue,draw=none] at (-5.5,-3) {\tiny $(-1,-1)$};
	\node at (0,-11.5) {\textit{Figure: Blocks}};
	\end{tikzpicture}
	\end{center}
\end{minipage}

For $d=1$, a graded generalized Eulerian $A_1(A)$-module has a rigidity property due to \cite[Theorem 6.1]{Put}, \cite{PutRoy22}. We now 
extend this property to our setting.
\begin{theorem}[Rigidity]\label{multi-rigid}
{\rm(}with hypothesis as in \ref{sa}{\rm)} For a nonempty subset $U$ of $\mathcal{S}$, take a block $\mathcal{B}(\underline{a}^U)$. Then the following are equivalent:
	\begin{enumerate}[\rm(a)]
		\item $M_{\underline{u}} \neq 0$ for all $\underline{u} \in \mathcal{B}(\underline{a}^U)$.
		\item $M_{\underline{w}} \neq 0$ for some $\underline{w} \in \mathcal{B}(\underline{a}^U)$. 
	\end{enumerate}
\end{theorem}

\begin{proof}
	Fix $i$. Consider the multigraded $A_{\underline{1}-e_i}(A)$-linear exact sequence
	\[0 \to H_1(\partial_i; M) \to M(e_i) \overset{\cdot \partial_i}{\lrt} M \to  H_0(\partial_i; M)\to 0,\]
	which induces an exact sequence 
	\[0 \to H_1(\partial_i; M)_{\underline{u}} \to M_{\underline{u}+e_i} \overset{\cdot \partial_i}{\lrt} M_{\underline{u}} \to  H_0(\partial_i; M)_{\underline{u}}\to 0\] 
	of $A$-modules. By Proposition \ref{genEur-partial}, $H_j(\partial_i; M)_{\underline{u}} \neq 0$ only if $u_i=-1$ for $j=0,1$. Thus the map $M_{\underline{u}+e_i} \overset{\cdot \partial_i}{\lrt} M_{\underline{u}}$ is an isomorphism for all $\underline{u} \in \Z^d$ with $u_i \leq -2$, that is, 
	\[M_{\underline{u}} \cong M_{(u_1, \ldots, u_{i-1}, -1, u_{i+1}, \ldots, u_d)} \quad \mbox{for all } \underline{u} \in \Z^d \mbox { with } u_i \leq -1.\]  
	Next, fix $j$. In view of Proposition \ref{genEur-XY}, one can show in a similar way that the map $M_{\underline{u}-e_j} \overset{\cdot X_j}{\lrt} M_{\underline{u}}$ is an isomorphism for all $\underline{u} \in \Z^d$ with $u_j \geq 1$, i.e., 
	\[M_{\underline{u}} \cong M_{(u_1, \ldots, u_{j-1}, 0, u_{j+1}, \ldots, u_d)} \quad \mbox{for all } \underline{u} \in \Z^d \mbox { with } u_j \geq 0.\] 
	It thus follows that given any subset $U$ of $\mathcal{S}$,
	\begin{equation}\label{comp_rel}
	M_{\underline{u}} \cong M_{\underline{a}^U} \quad \mbox{for all } \underline{u} \in \mathcal{B}(\underline{a}^U).
	\end{equation}
The result follows.
\end{proof}

\begin{remarks}
	
\noindent
\begin{enumerate}[\rm i)]
\item If $M \neq 0$ then $M_{w} \neq 0$ for some $\underline{w}\in \Z^d$. Since $\Z^d=\bigcup_{U \subseteq \mathcal{S}} \mathcal{B}(\underline{a}^U)$, there is a subset $U$ of $\mathcal{S}$ such that $\underline{w} \in \mathcal{B}(\underline{a}^U)$. By Theorem \ref{multi-rigid} it follows that $M_{\underline{u}} \neq 0$ for every $\underline{u} \in \mathcal{B}(\underline{a}^U)$.

\item We say an $A$-module $M$ is of \emph{rank $r$} if $M \otimes_A Q(A)$ is a free $Q(A)$-module of rank $r$, where $Q(A)$ denotes the \emph{total ring of fractions of $A$}, see \cite[Definition 1.4.2]{BruHer}. By Theorem \ref{multi-rigid}, we have
\begin{equation}\label{rank-rel}
\rank M_{\underline{u}} \cong \rank \mathcal{B}(\underline{a}^U) \quad \mbox{for all $\underline{u} \in \mathcal{B}(\underline{a}^U)$}.
\end{equation}
Suppose that $A$ is a domain. Then $Q(A)$ is in fact the field of fractions of $A$. In this case, \eqref{rank-rel} is already known. 
For instance, let $S=R \otimes_A Q(A) \cong Q(A)[X_1, \ldots, X_d]$. Then $\rank M_{\underline{u}}=\dim_{Q(A)} M_{\underline{u}} \otimes_A Q(A) \cong \dim_{Q(A)} N_{\underline{u}}$, where $N=H^i_{IS}(S)$. Observe that $IS$ is a usual monomial ideal of $S$. So by \cite[Definition 2.6, Theorem 2·11]{Yan2001},
\[N_{\underline{u}} \cong N_{\underline{a}^U} \quad \mbox{for all } \underline{u} \in \mathcal{B}(\underline{a}^U).\]
Hence $\dim_{Q(A)} N_{\underline{u}} \cong \dim_{Q(A)} N_{\underline{a}^U}$.
\end{enumerate}	
\end{remarks}

\section{Connection with straight modules}


In this section, $S=K[X_1, \ldots, X_d]$ is a polynomial ring in $d$ variables over a field $K$ of \emph{any characteristic}. Consider the standard $\Z^d$-grading on $S$ by setting $\deg X_1=e_i$ for $i=1, \ldots, r$. In \cite{Yan2001}, Yanagawa introduced the notion of {\it straight $S$-modules}.
For $\underline{w}\in \N^d$, we put $X^{\underline{w}}:=X_1^{w_1} \cdots X_d^{w_d}$.

\begin{definition}\label{str}
A $\Z^d$-graded $S$-module $M=\bigoplus_{\underline{u}\in \Z^d} M_{\underline u}$ is called straight, if the following two conditions are satisfied.
	
	(a) $\dim_K M_{\underline u}< \infty$ for all $\underline{u}\in \Z^d$.
	
	(b) The multiplication map $M_{\underline u}\ni y \mapsto X^{\underline{w}} \cdot y \in  M_{\underline{w}+\underline{u}}$ is bijective for all $\underline{u}\in \Z^d$ and each $\underline{w}\in \N^d$ with $\supp_+(\underline{w}+\underline{u}) = \supp_+(\underline{u})$.	
\end{definition} 
Note that condition (b) is equivalent to the condition: the multiplication map $M_{\underline u}\ni y \mapsto X_i \cdot y \in  M_{\underline{u}+e_i}$ is bijective for all $\underline{u}\in \Z^d$ with $u_i \neq 0$ for $i=1, \ldots, d$.

\vspace{0.15cm}
Suppose that $J \subseteq S$ is a monomial ideal. Since $H^i_J(S)=H^i_{\sqrt{J}}(S)$, without loss of generality we can assume that $J$ is square-free. The main example of a straight module is $H^i_J(S)(-\underline{1})$. In \cite{AlvGarZar}, J. A. Montaner, R. Garcia, and S. Zarzuela defined
a $\Z^d$-graded $S$-module $M$ 
to be a \emph{$\varepsilon$-straight module} if $M(-\underline{1})$ is a straight $S$-module.
The forthcoming result says that any $\Z^d$-graded generalized Eulerian $A_{\underline{1}}(K)$-module is $\varepsilon$-straight when $K$ is a field of characteristic zero.

\begin{lemma}[With hypothesis as in \ref{sa}]\label{genEu-imp-str}
Let $M$ be a $\Z^d$-graded generalized Eulerian $A_{\underline{1}}(A)$-module. Fix $\underline{w}\in \N^d$. Then the multiplication map $M_{\underline u}\in y \mapsto X^{\underline{w}}\cdot y \in  M_{\underline{w}+\underline{u}}$ is bijective for all $\underline{u}\in \Z^d$
with $\supp_+(\underline{w}+\underline{u}) = \supp_+(\underline{u})$.
	
Additionally, let $A=K$ be a field and $M$ be a holonomic $A_{\underline{1}}(K)$-module. Then $M(-\underline{1})$ is a straight $K[X_1, \ldots, X_d]$-module.	
\end{lemma}
\begin{proof}
	By \eqref{comp_rel},
	$M_{\underline{u}} \cong M_{\underline{a}^U}$ for all $\underline{u} \in \mathcal{B}(\underline{a}^U)$.
	Note that $\underline{a}^U+(\underline{1})=\underline{c}^U$ where $c_i=1$ if $i \in U$ and $c_i=0$ if $i \notin U$. Clearly, $M_{\underline{a}^U} \cong M(-\underline{1})_{\underline{c}^U}$. 
	We set
	\[\mathcal{B}'(\underline{c}^U):=\{\underline{u} \in \Z^d \mid \mbox{ for all } i=1, \ldots, d; ~u_i \geq 1 \mbox{ if } i \in U \mbox{ and } u_i \leq 0 \mbox{ if } i \notin U \}\]
	Then
	\[M(-\underline{1})_{\underline{u}} \cong M_{\underline{a}^U} \cong M(-\underline{1})_{\underline{c}^U} \quad \mbox{for all } \underline{u} \in \mathcal{B}'(\underline{c}^U).\]
	Notice $\Z^d=\bigcup_{U \subseteq \mathcal{S}}\mathcal{B}(\underline{a}^U)=\bigcup_{U \subseteq \mathcal{S}}\mathcal{B}'(\underline{c}^U)$. So it can be verified that
	$\supp_+(\underline{u})=\supp_+(\underline{v})$ if and only if $\underline{u}$, $\underline{v} \in \mathcal{B}'(\underline{c}^U)$ for some subset $U$ of $\mathcal{S}$. The result follows.
	
	Next, suppose that $A=K$ is a field and $M$ is a holonomic $A_{\underline{1}}(K)$-module. 
By Proposition \ref{hol-mon}, $\dim_K M_{\underline{u}} < \infty$ for every $\underline{u} \in \Z^d$. 	
	Consequently, $M(-\underline{1})$ is a straight module.
\end{proof}

\begin{center}
	\begin{tikzpicture}[scale=0.15]
	\draw[->] (-8.5,0)--(9,0) node[right]{$u$};
	\draw[->] (0,-8.5)--(0,9) node[above]{$v$};
	\draw[densely dotted, red] (-2,8.5)--(-2,-8.5);
	\draw[densely dotted, red] (-8.5,-2)--(8.5, -2);
	\draw[fill=cyan,fill opacity=0.35,draw=none] (0,8.5)--(0,0)--(8.5,0)--(8.5,8.5)--(0,8.5);
	\draw[fill=cyan,fill opacity=0.35,draw=none] (-2,8.5)--(-2,0)--(-8.5,0)--(-8.5,8.5)--(-2,8.5);
	\draw[fill=cyan,fill opacity=0.35,draw=none] (-2,-8.5)--(-2,-2)--(-8.5,-2)--(-8.5,-8.5)--(-2,-8.5);
	\draw[fill=cyan,fill opacity=0.35,draw=none] (0,-8.5)--(0,-2)--(8.5,-2)--(8.5,-8.5)--(0,-8.5);
	\node at (0,-2){\textcolor{red}{$\bullet$}};
	\node at (0,0){\textcolor{red}{$\bullet$}};
	\node at (-2,0){\textcolor{red}{$\bullet$}};
	\node at (-2,-2){\textcolor{red}{$\bullet$}};
	\node[blue,draw=none] at (3,-3) {\tiny $(0,-1)$};
	\node[blue,draw=none] at (2.5,1) {\tiny $(0,0)$};
	\node[blue,draw=none] at (-5,1) {\tiny $(-1,0)$};
	\node[blue,draw=none] at (-5.5,-3) {\tiny $(-1,-1)$};
	\node at (0,-11.5) {\textit{Figure: $M$}};
	\end{tikzpicture}	
	\hspace{2.5cm}
	\begin{tikzpicture}[scale=0.15]
	\draw[->] (-8.5,0)--(9,0) node[right]{$u$};
	\draw[->] (0,-8.5)--(0,9) node[above]{$v$};
	\draw[fill=cyan,fill opacity=0.35,draw=none] (2,8.5)--(2,2)--(8.5,2)--(8.5,8.5)--(2,8.5);
	\draw[blue] (2,8.5)--(2,2)--(8.5,2);
	\draw[fill=cyan,fill opacity=0.35,draw=none] (0,8.5)--(0,2)--(-8.5,2)--(-8.5,8.5)--(0,8.5);
	\draw[blue] (0,8.5)--(0,2)--(-8.5,2);
	\draw[fill=cyan,fill opacity=0.35,draw=none] (0,-8.5)--(0,0)--(-8.5,0)--(-8.5,-8.5)--(0,-8.5);
	\draw[blue] (0,-8.5)--(0,0)--(-8.5,0);
	\draw[fill=cyan,fill opacity=0.35,draw=none] (2,-8.5)--(2,0)--(8.5,0)--(8.5,-8.5)--(0,-8.5);
	\draw[blue] (2,-8.5)--(2,0)--(8.5,0);
	\node at (2,0){\textcolor{red}{$\bullet$}};
	\node at (0,0){\textcolor{red}{$\bullet$}};
	\node at (0,2){\textcolor{red}{$\bullet$}};
	\node at (2,2){\textcolor{red}{$\bullet$}};
	\node[blue,draw=none] at (-2.5,3) {\tiny $(0,1)$};
	\node[blue,draw=none] at (4.5,3) {\tiny $(1,1)$};
	\node[blue,draw=none] at (4.5,-1.2) {\tiny $(1,0)$};
	\node[blue,draw=none] at (-2.5,-1.2) {\tiny $(0,0)$};
	\node at (0,-11.5) {\textit{Figure: $M(-\underline{1})$}};
	\end{tikzpicture}
\end{center}

By $\omega_S$ we denote the canonical module of $S$. Recall that $\omega_S \cong S(-\underline{1})$. For the rest of this section, we assume that $\chars K=0$. Lemma \ref{genEu-imp-str} gives us an alternative proof of the following result.

\begin{theorem}[\cite{Mus}, \cite{Ter}]\label{LC_straight}
Let $J$ be a squarefree monomial ideal. For all $i \geq 0$, the local cohomology module $H^i_J(\omega_S) \cong H^i_J(S)(-\underline{1})$ is a straight module.
	
More generally, assume that $\mathcal{T}$ is a multigraded Lyubeznik functor on ${}^*\Mod(S)$. Then $\mathcal{T}(S)(-\underline{1})$ is a straight $S$-module.
\end{theorem}

An action of the $d^{th}$-Weyl algebra $A_{d}(K)$ 
on a straight $S$-module is defined in \cite[Remark 2.12]{Yan2001}. 

\begin{theorem}\label{str-GE}
Let $K$ is a field of characteristic zero. Let $M$ be a $\Z^d$-graded $S=K[X_1, \ldots, X_d]$-module. If $M$ is straight, then $M(\underline{1})$ is a $\Z^d$-graded Eulerian, holonomic $A_{\underline{1}}(K)$-module. Conversely, if $M(\underline{1})$ is a $\Z^d$-graded generalized Eulerian, \emph{holonomic} $A_{\underline{1}}(K)$-module, then $M$ is a straight $S$-module.	
\end{theorem}

\begin{proof}
	The converse follows from Lemma \ref{genEu-imp-str}. 
	
	Next, let $x \in M_{\underline{u}}$. If $u_i \neq 1$, then $M_{\underline{u}-e_i} \overset{\cdot X_i}{\lrt} M_{\underline{u}}$ is a bijective map. Thus there is an unique element $y \in M_{\underline{u}-e_i}$ such that $X_i \cdot y=x$. We set $\partial_i \cdot x=(u_i-1)y$. If $u_i=1,$ then we put $\partial_i \cdot x=0$. Since in any case $\partial_i \cdot x \in M_{\underline{u}-e_i}$ for all $i$, this action induces a $\Z^d$-graded structure on $M$ as an $A_{\underline{1}}(K)$-module. Note that $X_i \partial_i \cdot x= X_i \cdot (u_i-1)y= (u_i-1) X_i \cdot y=(u_i-1)x$. Set $N:=M(\underline{1})$. Then for any $z \in N_{\underline{v}}=M(\underline{1})_{\underline{v}}=M_{\underline{1}+ \underline{v}}$ we have 
	\[X_i \partial_i \cdot z= (\overline{1+v_i}-1)z=v_i z \qquad \mbox{for } i=1, \ldots, r.\]
	Hence $N$ is a $Z^d$-graded Eulerian $A_{\underline{1}}(K)$-module. 
	
	By \cite[Remark 2.12 (i)]{Yan2001}, any straight $S$-module $M$ is a holonomic $A_d(K)$-module. So 
$M$ is a holonomic $\Z^d$-graded $A_{\underline{1}}(K)$-module. 
This immediately establishes the result.
\end{proof}

\begin{remark}
From Theorem \ref{str-GE} it follows that any $\Z^d$-graded generalized Eulerian, \emph{holonomic} $A_{\underline{1}}(K)$-module is $\Z^d$-graded Eulerian.
\end{remark}

\section{Finiteness of some invariants linked to the components and a comparative study}

Throughout this section along with the hypothesis \ref{sa} we further assume that $A$ is regular. To obtain finiteness properties, we need the forthcoming results.

\begin{proposition}[with hypotheses as in \ref{sa}]\label{Koszul-local}
Suppose that $A=K[[Y_1, \ldots, Y_m]]$ and $\Lambda= A\langle \delta_1, \ldots, \delta_m\rangle$ is the ring of $K$-linear differential operators on $A$. If $L$ is a $\Z^d$-graded generalized Eulerian, holonomic  $A_{\underline{1}}(\Lambda)$-module, then $H_l(Y_m, L)$ are $\Z^d$-graded generalized Eulerian, holonomic $A_{\underline{1}}(\Lambda')$- modules for $l = 0, 1$ where $\Lambda'=A' \langle \delta_1, \ldots, \delta_{m-1}\rangle$ and $A'=K[[Y_1, \ldots, Y_{m-1}]]$.
\end{proposition}

\begin{proof}	
Notice $H_1(Y_m, L)=\{e \in L \mid Y_m \cdot e=0\}$ and $H_0(Y_m, L)=L/Y_mL$. Since the map $L \xrightarrow{\cdot Y_m} L$ is $A_{\underline{1}}(\Lambda')$-linear, $H_l(Y_m, L)$ are $\Z^d$-graded generalized Eulerian $A_{\underline{1}}(\Lambda')$-modules by Proposition \ref{exact}. Moreover, by Definition \ref{holo-multi} and \cite[Theorem 4.23]{Put}, they are holonomic $A_{\underline{1}}(\Lambda')$-modules for $l = 0, 1$.
\end{proof}


The following result can be established from Proposition \ref{Koszul-local} by using induction.
\begin{corollary}[with hypotheses as in Proposition \ref{Koszul-local}]\label{full-koszul-hom} Then for each $\nu \geq 0$, the Koszul homology $H_\nu(Y_1, \ldots, Y_m; L)$ is a holonomic, $\Z^d$-graded generalized Eulerian module over the Weyl algebra $A_{\underline{1}}(K)$.
\end{corollary}

Suppose that $L$ is an $A$-module. Let $\injdim_A L$ stand for the injective dimension of $L$.
The $j^{th}$ Bass number of $L$ with respect to a prime ideal $\p$ in $A$ is defined as 
\[\mu_j(\p, L) = \dim_{k(\p)} \Ext^j_{A_{\p}}(k(\p), L_\p),\] where $k(\p)$ denotes the residue field of $A_\p$.

The following result due to Lyubeznik is crucial for the study of Bass numbers.

\begin{lemma}\label{Lyu}\cite[1.4]{Lyu93}
Let $B$ be a Noetherian ring and $N$ be a (not necessarily finitely generated) $B$-module. Let $P$ be a prime ideal in $B$. If $\left(H^j_P(N)\right)_P$ is injective for all $j \geq 0$ then $\mu_j(P,N) = \mu_0\left(P,H^j_P(N)\right)$.
\end{lemma}

\begin{theorem}[with hypotheses as in \ref{sa}]\label{Bass_fin} Further assume that $A$ is regular. Fix $\underline{u} \in \Z^d$. Let $\p$ be a prime ideal in $A$. Then for each $j \geq 0$, the Bass number $\mu_j\left(\p, M_{\underline{u}}\right)$ is finite. 	
\end{theorem}

\begin{proof}
Fix $\underline{u} \in \Z^d$ and put $N:=H^i_I(R)_{\underline{u}}$. Take a prime ideal $\p$ in $A$ and fix $j \geq 0$.
\begin{claim*}
$H^j_\p(N)_\p$ is an injective $A$-module.
\end{claim*}
 Clearly, either $H^j_\p(N)_\p=0$ or $H^j_\p(N)_\p \neq 0$.
Since the first case is trivial, we assume $H^j_\p(N)_\p \neq 0$. This occurs when $\p$ is the minimal prime of $H^j_\p(N)$. Set $B:=\widehat{A_\p}$ and put $T:=\widehat{A_\p}[X_1, \ldots, X_d]$. As $B$ is a complete regular local ring containing a field of characteristic zero so by the \emph{Cohen-structure theorem}, $B \cong k[[Y_1, \ldots, Y_m]]$ where $k:=k(\p)$ is the residue field of $A_\p$ and $m: =\hgt \p$. Using similar arguments as in \cite[2.4]{Put} we get that $\mathcal{F}(-):=\widehat{A_\p} \otimes_A H^j_{\p R} \left(H^i_I(-)\right)$ is a $\Z^d$-graded Lyubeznik functor on ${}^*\Mod(T)$. Let $\D_B$ denote the ring of $k$-linear differential operators on $B$. Then by 
Corollary \ref{genEul-LyuFun}, $\FF(T)$ is a $\Z^d$-graded generalized Eulerian, holonomic $A_{\underline{1}}\left(\D_B\right)$-module. Therefore, $\FF(T)_{\underline{u}}$ is a 
$\D_B$-module. Since $\FF(T)_{\underline{u}}$ is supported only at the maximal ideal of $B$, by \cite[Proposition 2.3, Theorem 2.4]{Lyu93}, there is an ordinal (possibly infinite) $\alpha(\underline{u})$ such that $\FF(T)_{\underline{u}} \cong E_B(k)^{\alpha(\underline{u})}$, where $E_B(k)$ denotes the injective hull of $k$ as a $B$-module.
Notice $E_B(k) \cong E_A(A/\mathfrak{p})$. As $H^j_{\p} \left(M_{\underline{u}}\right)_\p$ has a natural structure of $\widehat{A_\p}$-module (see \cite[2.12]{Put}) so we have $H^j_{\p} \left(M_{\underline{u}}\right)_\p \cong \FF(T)_{\underline{u}}$.
The claim follows. 

Hence Lemma \ref{Lyu} is applicable, that is,
\begin{equation}\label{Bass_i_0}
\mu_j(\p, M_{\underline{u}})=\mu_0(\p, H^j_{\p R}\left(M\right)_{\underline{u}})=\mu_0(\p,\FF(T)_{\underline{u}}).
\end{equation}  
	
By Corollary \ref{full-koszul-hom} we have $\mathcal{V}=H_m(Y_1,\ldots, Y_m; \FF(T))$ is a $\Z^d$-graded generalized Eulerian, holonomic $A_{\underline{1}}(K)$-module with 
\[\mathcal{V}_{\underline{u}}=H_m(Y; \FF(T))_{\underline{u}}=H_m(Y; \FF(T)_{\underline{u}})=H_m(Y; E_B(K)^{\alpha_{\underline{u}}})=H_m(Y;E_B(K))^{\alpha_{\underline{u}}}=K^{\alpha_{\underline{u}}}.\]
The last equality holds by \cite[2.11]{Put}. From Theorem \ref{hol-mon} it follows that the Bass number \[\mu_j(P,M_{\underline{u}})=\alpha_{\underline{u}}=\dim_k \mathcal{V}_{\underline{u}}\] is finite.
\end{proof}

\begin{theorem}[with hypotheses as in Theorem \ref{Bass_fin}]\label{injdim-dim} 
For fixed $\underline{u} \in \Z^d$, 
\[\injdim M_{\underline{u}} \leq \dim M_{\underline{u}}.\]	
\end{theorem}

\begin{proof}	
From the \emph{Grothendieck vanishing theorem} we have $H^j_\p(M_{\underline{u}}) = 0$ for all $j > \dim M_{\underline{u}}$, see \cite[6.1.2]{BroSha}. In view of \eqref{Bass_i_0} we get that $\mu_j(\p, M_{\underline{u}})=\mu_0(\p, H^j_\p(M_{\underline{u}}))=0$ for all $j > \dim M_{\underline{u}}$. 
The statement follows.
\end{proof}

Given any subset $U$ of $\mathcal{S}$, from \eqref{comp_rel} we have the isomorphism $M_{\underline{u}} \cong M_{\underline{a}^U}$ of $A$-modules for all $\underline{u} \in \mathcal{B}(\underline{a}^U)$. As an immediate consequence we get the two successive results.

\begin{theorem}[with hypotheses as in Theorem \ref{Bass_fin}]\label{multi-bass} Let $\p$ be a prime ideal in $A$. Fix $j \geq 0$. EXACTLY one of the following holds:
		\begin{enumerate}[\rm (a)]
			\item $\mu_j(\p, M_{\underline{u}}) = 0$ for all $\underline{u} \in \Z^d$.
			\item There is a collection $\mathcal{Z}$ of subsets of $\mathcal{S}$ such that $\mu_j(\p, M_{\underline{u}}) \neq 0$ for every $\underline{u} \in \mathcal{B}(\underline{a}^U)$ with $U \in \mathcal{Z}$ and $\mu_j(\p, M_{\underline{u}}) = 0$ otherwise. Moreover, 
			\[\mu_j(\p, M_{\underline{u}}) = \mu_j(\p, M_{\underline{a}^U}) \quad \mbox{ for all } \underline{u} \in \mathcal{B}(\underline{a}^U).\]
		\end{enumerate}
\end{theorem}
Let $\Supp_A L = \{\p \mid \p \mbox{ is a prime in } A \mbox{ and } L_\p \neq 0\}$ be the support of an $A$-module $L$. By $\dim_A L$ we mean the dimension of $\Supp_A L$ as a subspace of $\Spec(A)$, the set of prime ideals in $A$. By \eqref{comp_rel}, we obtain the following.

\begin{theorem}[with hypotheses as in Theorem \ref{Bass_fin}]\label{multi-injdim-dim} Fix a subset $U$ of $\mathcal{S}$. Then 
	\[\injdim M_{\underline{u}} = \injdim M_{\underline{a}^{U}} \quad \mbox{and} \quad \dim M_{\underline{u}} = \dim M_{\underline{a}^{U}}\] 
	for all $\underline{u} \in \mathcal{B}\left(\underline{a}^U\right)$.
\end{theorem}

\s{\bf Associated primes:}\\
For fixed $i\geq 0$, we 
now study the behavior of the set $\Ass_A H^i_I(R)_{\underline{u}}$ of associated primes of $H^i_I(R)_{\underline{u}}$ as $\underline{u} \in \Z^d$ varies.

\begin{theorem}[with hypotheses as in \ref{sa}]\label{multi-ass}
	Further assume that either $A$ is regular local or a smooth affine algebra over a field of characteristic zero. Then 
	\begin{enumerate}[\rm (1)]
		\item $\bigcup_{\underline{u} \in \Z^d} \Ass_A M_{\underline{u}}$ is a finite set.
		\item For any subset $U \subseteq \mathcal{S}$, \[\Ass_A M_{\underline{u}} = \Ass_A M_{\underline{a}^U} \quad  \mbox{ for all }\underline{u} \in \mathcal{B}(\underline{a}^U).\]
	\end{enumerate} 
\end{theorem}

To prove the foregoing theorem we need the following result from \cite{Put}.
\begin{proposition}\label{fini-ass-ringext}
Let $f \colon C \rt B$ be a homomorphism of Noetherian rings. Let $M$ be a $B$-module. Then 
\[\Ass_C M  = \{ P\cap C \mid P \in \Ass_B M \}.\]
In particular, if $\Ass_B M$ is a finite set then so is $\Ass_C M$.	
\end{proposition}

\begin{proof}[Proof of Theorem \ref{multi-ass}]
We first show that under the given hypotheses $\Ass_R \TT(R)$ is finite.
	
If $A$ is a smooth affine algebra over a field then so is $R = A[X_1,\ldots, X_d]$. In this case, if $\mathcal{G}$ is \emph{any} Lyubeznik functor on $Mod(R)$ (not necessarily graded) then $\Ass_R \mathcal{G}(R)$ is finite, see \cite[3.7]{Lyu93}.
	
Next, assume that $A$ is local with maximal ideal $\n$. Let $\M = (\n, X_1,\ldots, X_d)$ be the maximal multihomogeneous ideal of $R$. Since $\TT(R)$ is a multigraded $R$-module, all its associate primes are multihomogeneous (see \cite[1.5.6]{BruHer}). So they are contained in $\M$. Thus we have an isomorphism $\Ass_R(\TT(R)) \xrightarrow{\cong} \Ass_{R_\M}(\TT(R)_\M)$. However,
$\TT(R)_{\M} = \mathcal{G}(R_\M)$ for a Lyubeznik functor $\mathcal{G}(-)$ on $Mod(R_\M)$. As $R_\M$ is a regular local ring so the finiteness of $\Ass_R \TT(R)$ follows from \cite[3.3]{Lyu93}.
	
In view of the above, (1) follows from Proposition \ref{fini-ass-ringext}.

\vspace{0.15cm}	
(2) is an immediate consequence of \eqref{comp_rel}.
\end{proof}

\s {\bf Infinite generation:}

In this subsection, we give a sufficient condition for infinite generation of multigraded components of $H^i_I(R)$. 

\begin{theorem}[with hypothesis as in Theorem \ref{injdim-dim}]\label{non-fg}
Further assume that $A$ is a domain
and $I \cap A \neq 0$. 
If $H^i_I(R)_{\underline{u}} \neq 0$, then $H^i_I(R)_{\underline{u}}$ is not finitely generated as an $A$-module.
\end{theorem}
\begin{proof}
Let $P$ be a prime ideal in $A$. Note that $R_P=A_P[X_1, \ldots, X_d]$ and $(M_{\underline{u}})_P= H_I^i(R)_{\underline{u}} \otimes _A A_P= H_I^i(R \otimes_A A_P)_{\underline{u}}= H_I^i(R_P)_{\underline{u}}$. Clearly if $H_I^i(R_P)_{\underline{u}}$ is not a finitely generated $A_P$-module, then $M_{\underline{u}}$ is also not a finitely generated $A$-module. Therefore it is enough to prove the result considering $A$ is a local ring.
	
If possible, suppose that $M_{\underline{u}}$ is finitely generated as an $A$-module. From Theorem \ref{multi-injdim-dim} we have $\injdim M_{\underline{u}}\leq \dim M_{\underline{u}}$.
So by \cite[Theorem 3.1.17]{BruHer}, 
$\dim M_{\underline{u}} \leq \injdim M_{\underline{u}}= \depth A$. Since $A$ is \CM, $\depth A=\dim A$. Together we get that 
\[\dim M_{\underline{u}} \leq \dim A= \injdim M_{\underline{u}} \leq \dim M_{\underline{u}}.\] 
Thus $\dim M_{\underline{u}}= \dim A$. 
As $A$ is a domain so the zero ideal is an associated prime of $M_{\underline{u}}$.
Therefore, $M_{\underline{u}}$ has a nonzero torsion-free element. Nevertheless, $M$ is $I$-torsion and hence $M_{\underline{u}}$ is $(I\cap A)$-torsion. Consequently, the hypothesis $I \cap A \neq 0$ leads to a contradiction.
\end{proof}

If $I\subseteq R$ is a usual monomial ideal, then the finite generation of the components of $H^i_I(R)$ is expected to persist. However, the following result says even more about their structure.

\begin{proposition}[With hypotheses as in \ref{sa}]
Let $I$ be a usual monomial ideal. Then $H^i_I(R)_{\underline{u}}$ is a finite free $A$-module.
\end{proposition}

\begin{proof}
Suppose that $\{f_1, \ldots, f_t\}$ is a monomial generating set of the ideal $I$. Recall that $A$ contains a field $K$ of characteristic zero. Consider the monomial ideal $J=(f_1, \ldots, f_t)$ in the subring $K[X_1, \ldots, X_d]$ of $R$. Clearly, $S:=K[X_1, \ldots, X_d] \hookrightarrow R$ is a flat extension. It is 
apparent that $H^i_I(R)\cong H^i_{JR}(R) \cong H^i_J(S) \otimes_S R$ and hence $H^i_I(R)_{\underline{u}} \cong H^i_J(S)_{\underline{u}} \otimes_k A$ for every $\underline{u} \in \Z^d$. Since $H^i_J(S)_{\underline{u}}$ is a finite dimensional $k$-vector space (see \cite[Theorem 2.1]{Mus}, \cite{Ter}, or \cite[Corollary 3.3]{Yan2001}), it follows that $H^i_I(R)_{\underline{u}} \cong k^r \otimes_k A \cong A^r$ for some finite $r>0$. Thus $H^i_I(R)_{\underline{u}}$ is a finite free $A$-module.
\end{proof}

\section{Multiplicity}

In this section, $A=K[[Y_1, \ldots, Y_m]]$ is the ring of formal power series in $m$ variables over a field $K$ of characteristic zero. In addition, $R=A[X_1, \ldots, X_d]$ is a polynomial ring over $A$ in $d$ variables. We consider the standard multrigrading on $R$. Let $\Lambda:=D_K(A)=A\langle \delta_1, \ldots, \delta_m\rangle$ be the ring of $K$-linear differential operator on $A$, where $\delta_i=\partial/\partial Y_i$ for $i=1, \ldots, m$. We set $\D:=A_{\underline{1}}(\Lambda),$ the Weyl algebra over $\Lambda$. Let $M=\bigoplus_{{\underline{u}} \in \Z^d} M_{\underline{u}}$ be a $\Z^d$-graded holonomic $\D$-module. Along the line of Theorem \ref{hol-mon}, one can show that $\D_{\underline{0}}=\Lambda[\e_1, \ldots, \e_d]$ and $M_{\underline{u}}$ is a finitely generated $\D_{\underline{0}}$-module for each $\underline{u} \in \Z^d$. 
Finite generation of $M_{\underline{u}}$ as a $\D_{\underline{0}}$-module also follows from \ref{comp_fil}. Consider the filtration $\mathcal{F}=\{\mathcal{F}_\nu\}_{\nu \geq 0}$ of finitely generated $\mathbb{Z}^d$-graded $R$-submodules of $\D$, where 
\[\mathcal{F}_{\nu}:=R \cdot \{\delta^{\underline{a}} \partial^{\underline{b}}: |\underline{a}|+|\underline{b}|\leq \nu\}.\]

\s \label{D-fil} \emph{A filtration on $\D_{\underline{u}}$}. \\
For $\underline{u}\in \Z^d$, we denote the $\underline{u}^{th}$-degree component of $\mathcal{F}_\nu$ by $\mathcal{F}_{\underline{u},\nu}$. Note that 
\begin{align*}
&\mathcal{F}_{\underline{0},0}=R_0=A\\
&\mathcal{F}_{\underline{0},1}=R_0+\sum_{j=1}^m\left(R \cdot \delta_j\right)_0+\sum_{i=1}^d\left(R \cdot \partial_i\right)_0=A+\sum_{j=1}^m A \cdot \delta_j+\sum_{i=1}^d A \cdot \mathcal{E}_i\\
&\mathcal{F}_{\underline{0},2}=A+\sum_{j=1}^m A \cdot \delta_j+\sum_{1\leq j_1, j_2 \leq m} A \cdot \delta_{j_1}\delta_{j_2}+\sum_{i=1}^d A \cdot \mathcal{E}_i+\sum_{1 \leq i_1, i_2 \leq d} A \cdot \mathcal{E}_{i_1}\mathcal{E}_{i_2} + \sum_{\substack{1\leq j \leq m\\ 1 \leq i \leq d}} A \cdot \delta_j \mathcal{E}_i.
\end{align*}
and so on. 

\begin{claim}
For every pairs $\nu, \mu \geq 0$ and $\underline{u}, \underline{v} \in \Z^d$, 
\begin{equation}\label{deg_inc}
\mathcal{F}_{\underline{u},\nu}\mathcal{F}_{\underline{v},\mu} \subseteq \mathcal{F}_{\underline{u}+\underline{v},\nu+\mu}.
\end{equation} 	
\end{claim}

Take $\xi \in \mathcal{F}_{\underline{u},\nu}$. Then $\xi=\sum C_{\underline{b}} \delta^{\underline{a}} \partial^{\underline{b}}$ such that $|\underline{a}|+|\underline{b}|\leq \nu$ and $C_{\underline{b}} \in R_{\underline{u}+\underline{b}}$. Similarly, take $\eta=\sum C_{\underline{d}} \delta^{\underline{c}} \partial^{\underline{d}} \in \mathcal{F}_{\underline{v},\mu}$, that is, $|\underline{c}|+|\underline{d}|\leq \mu$ and $C_{\underline{d}} \in R_{\underline{v}+\underline{d}}$. Suppose that  $C_{\underline{b}}=c_{\underline{b}}X^{\underline{u}+\underline{b}}$ and $C_{\underline{d}}=c_{\underline{d}}X^{\underline{v}+\underline{d}}$ for some $c_{\underline{b}}, c_{\underline{d}} \in A$. Note 
{\small
\[\partial_iX_i^b=(1+X_i\partial_i)X_i^{b-1}=X_i^{b-1}+X_i(1+X_i\partial_i)X_i^{b-2}=2X_i^{b-1}+X_i^2\partial_iX_i^{b-2}=\cdots=bX_i^{b-1}+X_i^b\partial_i.\]
}

\noindent
Thus for $b \geq a \geq 0$,
\begin{align*}
\partial_i^aX_i^b&=\partial_i^{a-1}\left(bX_i^{b-1}+X_i^b\partial_i\right)\\
&=b\partial_i^{a-2}\left((b-1)X_i^{b-2}+X_i^{b-1}\partial_i\right)+\partial_i^{a-2}\left(bX_i^{b-1}+X_i^{b}\partial_i\right)\partial_i\\
&=b(b-1)\partial_i^{a-2}X_i^{b-2}+(b+1)\partial_i^{a-2}X_i^{b-1}\partial_i+\partial_i^{a-2}X_i^{b}\partial_i^2\\
&=b(b-1)\partial_i^{a-3}\left((b-2)X_i^{b-3}+X_i^{b-2}\partial_i\right)+(b+1)\partial_i^{a-3}\left((b-1)X_i^{b-2}+X_i^{b-1}\partial_i\right)\partial_i\\
&\hspace{2cm}+\partial_i^{a-3}\left(bX_i^{b-1}+X_i^b\partial_i\right)\partial_i^2\\
&=b(b-1)(b-2)\partial_i^{a-3}X_i^{b-3}+[b(b-1)+(b+1)(b-1)]\partial_i^{a-3}X_i^{b-2}\partial_i\\
&\hspace{2cm}+[(b+1)+b]\partial_i^{a-3}X_i^{b-1}\partial_i^2+\partial_i^{a-3}X_i^{b}\partial_i^3\\
&=\cdots\\
&=\sum_{j=0}^a b_j X_i^{b-a+j}\partial_i^{j},
\end{align*}
where $b_j=f_j(b) \in \Z$. Recall $\delta_j$ commutes with both $X_i, \partial_i$ for $i=1, \ldots, d$ and $j=1, \ldots, m$. Using the same arguments as above we get $\delta_j^cY_j^d=\sum_{k=0}^c d_k Y_j^{d-c+k}\delta_j^{k}$ for $c \leq d$ and $\delta_j^cY_j^d=\sum_{k=0}^d c_k Y_j^{k}\delta_j^{c-d+k}$ for $c \geq d$ for $d_k=f_k(d), c_k=g_k(c)\in \Z$. Therefore, 
\[\left(C_{\underline{b}} \delta^{\underline{a}} \partial^{\underline{b}}\right) \cdot \left(C_{\underline{d}} \delta^{\underline{c}} \partial^{\underline{d}}\right)=\left(c_{\underline{b}}X^{\underline{u}+\underline{b}} \delta^{\underline{a}} \partial^{\underline{b}}\right) \cdot \left(c_{\underline{d}}X^{\underline{v}+\underline{d}} \delta^{\underline{c}} \partial^{\underline{d}}\right)\subset \mathcal{F}_{\underline{u}+\underline{v},\nu+\mu}.\]
The claim follows. 

\vspace{0.2cm}
Fix $\underline{u}\in \Z^d$. Then the inclusions 
\begin{enumerate}
\item $\mathcal{F}_\mu \subset \mathcal{F}_\nu$ for $\mu \leq \nu$,
\item $\mathcal{F}_\mu \mathcal{F}_\nu \subseteq \mathcal{F}_{\mu+\nu}$ for all $\mu, \nu \geq 0$,
\end{enumerate}
of $\Z^d$-graded $R$-modules (where the $\Z^d$-grading on $\mathcal{F}_\mu$ is inherited from $\D$) induces inclusions $\mathcal{F}_{\underline{u},\mu} \subset  \mathcal{F}_{\underline{u},\nu}$ and $\mathcal{F}_{\underline{u},\mu} \mathcal{F}_{\underline{u},\nu} \subseteq \mathcal{F}_{\underline{u},\mu+\nu}$ of finitely generated $R_{\underline{0}}=A$-modules. Moreover, $\D_{\underline{u}}=\bigcup_{\nu \geq 0} \FF_{\underline{u}, \nu}$, since $\D=\bigcup_{\nu \geq 0}\FF_{\nu}$. Thus $\FF_{\underline{u},*}=\{\FF_{\underline{u}, \nu}\}_{\nu \geq 0}$ forms a filtration on $\D_{\underline{u}}$ for every $\underline{u} \in \Z^d$. The associated graded ring corresponding to the filtration $\FF_{\underline{0},*}$ on $\D_{\underline{0}}$ is
\begin{equation}\label{polynomialRing}
\gr_{\FF_{\underline{0},*}}(\D_{\underline{0}})=\FF_{\underline{0},0} \oplus \frac{\FF_{\underline{0},1}}{\FF_{\underline{0},0}} \oplus \frac{\FF_{\underline{0},2}}{\FF_{\underline{0},1}} \oplus \cdots= A[\overline{\delta_1}, \ldots, \overline{\delta_m},\overline{\e_1}, \ldots, \overline{\e_d}],
\end{equation}
a polynomial ring over $A$ in $(m+d)$-variables with $\deg \overline{\delta_i}=1, \deg \overline{\e_j}=1$ for $i=1, \ldots, m$ and $j=1, \ldots, d$.

\s \label{comp_fil} \emph{An induced filtration on the components of $\Z^d$-graded $\D$-modules}. \\
Let $L=\oplus_{\underline{u} \in \Z^d}L_{\underline{u}}$ be a $\Z^d$-graded $\D$-module. A set $\Gamma:=\{\Gamma_{\nu}\}_{\nu \in \Z}$ of 
$R$-submodules of $L$ is called a \emph{$\Z^d$-graded $\mathcal{F}$-compatible filtration} on $L$ if the following holds:
\begin{enumerate}
\item $\Gamma_{\nu}$ is a finitely generated $\Z^d$-graded 
$R$-submodule of $L$ for each $\nu \in \Z$,
\item $\Gamma_\nu \subset \Gamma_{\nu+1}$ for all $\nu \in \Z$,
\item $\mathcal{F}_{\mu}\Gamma_\nu \subset \Gamma_{\mu +\nu}$ for every pair $\mu, \nu \in \Z$,
\item $\Gamma_\nu=0$ for $\nu \ll 0$,
\item $L=\bigcup_{\nu \in \Z} \Gamma_\nu.$
\end{enumerate}
We put $\Gamma(\nu)=\Gamma_\nu/\Gamma_{\nu-1}$ for each $\nu \in \Z$. Observe that $\gr_\Gamma(L)=\oplus_{\nu \in \Z}\Gamma(\nu)$ is a graded $\gr_\FF(\D)$-module. Besides, $\gr_\Gamma(L)$ has an induced $\Z^{d+1}$-graded structure, where for all $\underline{u}\in \Z^d$ and each $\nu \in \Z$,
\[\gr_\Gamma(L)_{\underline{u}, \nu}=\Gamma(\nu)_{\underline{u}} \cong \frac{(\Gamma_\nu)_{\underline{u}}}{(\Gamma_{\nu-1})_{\underline{u}}}.\] 
We say that $\Gamma$ is a good filtration if $\gr_\Gamma(L)$ is a finitely generated $\gr_\FF(\D)$-module. Using the same arguments as in \cite[Proposition 2.6, Proposition 2.7, Chapter 1]{Bjo}, one can show that a $\Z^d$-graded $\D$-module can be equipped with a good filtration if and only if it is finitely generated.

Fix $\underline{u}\in \Z^d$. Given a $\mathcal{F}$-compatible good filtration $\Gamma$ on $L$, we can define a filtration $\Gamma_{\underline{u},*}:=\{\Gamma_{\underline{u}, \nu}\}_{\nu \in \Z}$ on $M_{\underline{u}}$ 
by setting $\Gamma_{\underline{u}, \nu}:=(\Gamma_{\nu})_{\underline{u}}$, the $u^{th}$-degree component of $\Gamma_\nu$. Clearly, $M_{\underline{u}}=\bigcup_{\nu \geq 0} \Gamma_{\underline{u}, \nu}$. As $\Gamma_\nu \subset \Gamma_{\nu+1}$ is an inclusion of finitely generated $\Z^d$-graded $R$-submodules of $M$ so $\Gamma_{\underline{u}, \nu} \subset \Gamma_{\underline{u}, \nu+1}$ is an inclusion of finitely generated $A$-submodules of $M_{\underline{u}}$. Moreover,
\[\FF_{\underline{0}, \mu}\Gamma_{\underline{u}, \nu}=(\FF_\mu)_{\underline{0}}(\Gamma_{\nu})_{\underline{u}}\subset \FF_\mu\Gamma_{\nu} \cap \D_{\underline{0}}M_{\underline{u}} \subset \Gamma_{\mu+\nu} \cap M_{\underline{u}}=(\Gamma_{\mu+\nu})_{\underline{u}}=\Gamma_{\underline{u}, \mu+\nu}.\]
Hence $\Gamma_{\underline{u},*}$ is a $\FF_{\underline{0},*}$-compatible filtration on $M_{\underline{u}}$. Thus  $\gr_{\Gamma_{\underline{u},*}}(L_{\underline{u}})$ is a graded $\gr_{\FF_{\underline{0},*}}(\D_{\underline{0}})$-module. We now rewrite
\[\gr_\Gamma(L)=\bigoplus_{\underline{u}\in \Z^d}\left(\bigoplus_{\nu \in \Z} \Gamma(\nu)_{\underline{u}}\right)\]
and consider it as a $\Z^d$-graded module by setting $ \gr_\Gamma(L)_{\underline{u}}=\bigoplus_{\nu \in \Z} \Gamma(\nu)_{\underline{u}}$ for all $\underline{u}\in \Z^d$. As $\Gamma$ is a $\mathcal{F}$-compatible good filtration on $L$ so $\gr_\Gamma(L)$ is a finitely generated $\Z^d$-graded $\gr_\FF(\D)$-module. Hence $\gr_\Gamma(L)_{\underline{u}}\cong \gr_{\Gamma_{\underline{u},*}}(L_{\underline{u}})$ is a finitely generated $\gr_\FF(\D)_{\underline{0}} \cong \gr_{\FF_{\underline{0},*}}(\D_{\underline{0}})$-module. Thus 
$\Gamma_{\underline{u},*}$ is a $\FF_{\underline{0},*}$-compatible good filtration on $L_{\underline{u}}$. Therefore, using the similar arguments as in \cite[Proposition 2.4, Chapter 1]{Bjo}, we get that $L_{\underline{u}}$ is a finitely generated $\D_{\underline{0}}$-module for each $\underline{u} \in \Z^d$. 

\s\label{dim-mul} \emph{Dimensions and multiplicities}. \\
Let $S =\oplus_{\nu \geq 0}S_\nu$ be a Noetherian graded ring such that $S_0$ is a local ring with maximal ideal $\m_0$. By $\mathfrak{M}=(\m_0, S_+)$ we denote the unique homogeneous maximal ideal in $S$. Let $E$ be a finitely generated graded $S$-module. Then by \cite[Theorem 1.5.8]{BruHer},
\[\dim_S E=\dim_{S_{\mathfrak{M}}} E_{\mathfrak{M}}.\]

Suppose that $L$ is a $\Z^d$-graded $\D$-module and $\Gamma$ is a $\Z^d$-graded $\mathcal{F}$-compatible good filtration on $L$. By \cite[4.6]{Put}, $\gr_{\mathcal{F}}(\D) \cong R[\overline{\delta_1}, \ldots, \overline{\delta_m},\overline{\partial_1}, \ldots, \overline{\partial_d}]$ is a Noetherian ring. 
The finite generation of $\gr_{\mathcal{T}_{\underline{u}, \star}}(L_{\underline{u}})$ as a  $\gr_{\FF_{\underline{u},\star}}(\D_{\underline{0}})$-module obtained in \ref{comp_fil}, allows us to write
\begin{equation}\label{dim}
\dim_{\gr_{\FF_{\underline{u},\star}}(\D_{\underline{0}})} \gr_{\mathcal{T}_{\underline{u}, \star}}(L_{\underline{u}})=\dim_{\gr_{\FF_{\underline{u},\star}}(\D_{\underline{0}})_{\mathfrak{m}}} \gr_{\mathcal{T}_{\underline{u}, \star}}(L_{\underline{u}})_{\mathfrak{m}},
\end{equation}
where $\m=(Y_1, \ldots, Y_m)+(\overline{\delta_1}, \ldots, \overline{\delta_m},\overline{\e_1}, \ldots, \overline{\e_d})$. Fix $\underline{u} \in \Z^d$ and set $N_{\underline{u}}:=\gr_{\mathcal{T}_{\underline{u}, \star}}(L_{\underline{u}})_{\mathfrak{m}}$.
Since $\big(\gr_{\FF_{\underline{u},\star}}(\D_{\underline{0}})_{\mathfrak{m}}, \m\big)$ is a local ring with the residue field $K$, from \cite[Proposition 4.6.2]{BruHer} we get that the \emph{Hilbert Samuel function}
\[\chi_{N_{\underline{u}}}^\m(n)=\sum_{i=0}^n \dim_K \frac{\m^i N_{\underline{u}}}{\m^{i+1} N_{\underline{u}}}=\dim_K \frac{N_{\underline{u}}}{\m^{n+1}N_{\underline{u}}}\]
is of polynomial type of degree $d_M(\underline{u})=\dim_{\gr_{\FF_{\underline{u},\star}}(\D_{\underline{0}})} N_{\underline{u}}$ and the \emph{multiplicity} of $M_{\underline{u}}$ is 
\begin{equation}\label{multiplicity}
\mathrm{e}(M_{\underline{u}})=\mathrm{e}(N_{\underline{u}})=\lim_{n \to\infty} \frac{d(\underline{u})!}{n^{d(\underline{u})}}\chi_{N_{\underline{u}}}^\m(n).
\end{equation}
We call $\mathrm{d}_M(\underline{u})$ as the \emph{Bernstein-type dimension} of $M_{\underline{u}}$ for $\underline{u} \in \Z^d$.

Suppose that $\chi_{N_{\underline{u}}}^\m(n)=P(n)$ for $n \gg 0$, where $P(X)=\sum_{j=0}^{d(\underline{u})}a_j X^j \in \Q[X]$. Then note that $\mathrm{e}(N_{\underline{u}})=d(\underline{u})! a_{d(\underline{u})}$. Along the line of \cite[Lemma 6.2, Chapter 2]{Bjo}, one can further show that the definitions of Dimension and Multiplicity are independent of the choice of $\FF$-compatible good filtration. 

\s Let $M=\oplus_{\underline{u} \in \Z^d}M_{\underline{u}}$ be a finitely generated $\Z^d$-graded $\D$-module. Let $\{m_1, \ldots, m_s\}$ be a finite set of multihomogeneous generators of $M$ as a $\D$-module. 
We put $|m_i|:=\deg m_i$ and set 
\begin{equation}\label{T}
\mathcal{T}_{\nu}:=\mathcal{F}_{\nu}m_1+\cdots+\mathcal{F}_{\nu}m_s.
\end{equation} 
Observe that (i) $\mathcal{T}_\nu=0$ for $\nu <0$, (ii) $\mathcal{T}_0\subset \mathcal{T}_1 \subset \cdots$ is an increasing sequence of finitely generated
$\Z^d$-graded $R$-submodules of $M$, and (iii) $M_{\underline{u}}=\bigcup_{\nu \geq 0} \mathcal{T}_{\underline{u}, \nu}$. 
Since $\mathcal{F}_\mu \mathcal{F}_\nu \subseteq \mathcal{F}_{\mu+\nu}$ for all $\mu, \nu \geq 0$, it follows from \eqref{T} that $\mathcal{T}=\{\mathcal{T}_\nu\}$ is an $\mathcal{F}$-compatible filtration.  
In light of \ref{comp_fil}, we have an induced filtration $\mathcal{T}_{\underline{u}, \star}:=\{\mathcal{T}_{\underline{u}, \nu}\}_{\nu \geq 0}$ on $M_{\underline{u}}$, where 
\begin{equation}\label{fg}
\mathcal{T}_{\underline{u},\nu}:=\mathcal{F}_{\underline{u}-|m_1|,\nu}~m_1+\cdots+\mathcal{F}_{\underline{u}-|m_s|,\nu}~m_s.
\end{equation} 

Fix $\underline{u} \in \Z^d$. Given a pair $\mu, \nu \geq 0$, 
by \eqref{deg_inc}, $\mathcal{F}_{\underline{0},\mu} \mathcal{F}_{\underline{u}-|m_i|,\nu} \subset \mathcal{F}_{\underline{u}-|m_i|,\nu+\mu}$ for $i=1, \ldots, s$ and hence $\mathcal{F}_{\underline{0},\mu} \mathcal{T}_{\underline{u},\nu} \subset \mathcal{T}_{\underline{u},\mu+\nu}$. Therefore, $\mathcal{T}_{\underline{u}, \star}$ is an $\mathcal{F}_{\underline{0}, \star}$-compatible filtration. This statement also follows from \ref{comp_fil}.

In light of the proof of \cite[Proposition 2.7, Chapter 1]{Bjo}, it follows that $\mathcal{T}$ is a good filtration on $M$. So from the discussions in \ref{comp_fil} we get that $\mathcal{T}_{\underline{u}, \star}$ is a good filtration on $M_{\underline{u}}$. For the convenience of the readers, we give a proof here.

\begin{lemma}\label{goodFil}
$\gr_{\mathcal{T}_{\underline{u}, \star}}(M_{\underline{u}})$ is a finitely generated graded $\gr_{\FF_{\underline{0},\star}}(\D_{\underline{0}})$-module.
\end{lemma}

\begin{proof} 
We first show $\gr_{\mathcal{T}_{\underline{u}, \star}}(M_{\underline{u}})$ is a graded $\gr_{\FF_{\underline{0},\star}}(\D_{\underline{0}})$-module utilizing the fact that $\mathcal{T}_{\underline{u}, \star}$ is an $\mathcal{F}_{\underline{0}, \star}$-compatible filtration. For $\alpha \in \mathcal{T}_{\underline{u}, \nu}/\mathcal{T}_{\underline{u}, \nu-1}$ and $\beta \in\mathcal{F}_{\underline{0}, \mu}/\mathcal{F}_{\underline{0}, \mu-1}$, choose $m \in \mathcal{T}_{\underline{u}, \nu}$ and $f \in \mathcal{F}_{\underline{0}, \mu}$ such that their respective images are $\overline{m}=\alpha$ and $\overline{f}=\beta$. Since $f \cdot m \in \mathcal{F}_{\underline{0}, \mu} \mathcal{T}_{\underline{u}, \nu} \subset \mathcal{T}_{\underline{u}, \mu+\nu}$, we define $\beta \cdot \alpha$ to be the image of $fm$ in $\mathcal{T}_{\underline{u}, \mu+\nu}/\mathcal{T}_{\underline{u}, \mu+\nu-1}$. One can verify that the action is well-defined by using the same arguments given in \cite[2.3, Chapter 1]{Bjo}. 
This establishes our claim.

In light of \eqref{fg}, the set of images of $m_i$'s in $\mathcal{T}_{\underline{u},1}/\mathcal{T}_{\underline{u},0}$ generates 
$\gr_{\mathcal{T}_{\underline{u}, \star}}(M_{\underline{u}})$ as a $\gr_{\FF_{\underline{u},\star}}(\D_{\underline{0}})$-module. The result follows. 
\end{proof}

\s We devote the rest of this section in analyzing the asymptotic behaviour of Bernstein-type dimensions and multiplicities of $M_{\underline{u}}$, defined in \ref{dim-mul}. 

\vspace{0.2cm}
Fix $i \geq 0$. Observe that the $A$-linear map $M_{\underline{u}} \overset{\cdot X_i}{\lrt} M_{\underline{u}+e_i}$ sends
$\mathcal{T}_{\underline{u}, \nu}=\sum_{j=1}^s\mathcal{F}_{\underline{u}-|m_j|, \nu}~m_j$ to $X_i \cdot \mathcal{T}_{\underline{u}, \nu}= \sum_{j=1}^s X_i \cdot \mathcal{F}_{\underline{u}-|m_j|, \nu}~m_j=\sum_{j=1}^s\mathcal{F}_{\underline{u}-|m_j|, \nu} ~(X_i \cdot m_j)$. Clearly, $X_i \cdot \mathcal{F}_{\underline{u}-|m_j|, \nu} \subseteq \mathcal{F}_{\underline{u}+e_i-|m_j|, \nu}$ and hence $X_i \cdot \mathcal{T}_{\underline{u}, \nu} \subseteq  \mathcal{T}_{\underline{u}+e_i, \nu}$.

\begin{lemma}
	The induced map 
	\begin{equation*}\label{iso-gr-X}
	\gr_{\mathcal{T}_{\underline{u},\star}}(M_{\underline{u}}) \xrightarrow{\cdot X_i} \gr_{\mathcal{T}_{\underline{u}+e_i,\star}}(M_{\underline{u}+e_i})
	\end{equation*}
	is $\gr_{\FF_{\underline{0},\star}}(\D_{\underline{0}})$-linear for $i=1, \ldots, d$.
\end{lemma}

\begin{proof} 
Notice $Y_k X_i=X_i Y_k, \delta_k X_i=X_i \delta_k$ for $k=1, \ldots, m$ and $\e_j X_i=X_i \e_j$ for $j \neq i$. Now $\e_i \mathcal{T}_{\underline{u}, \nu}$ maps to $X_i \cdot \e_i\mathcal{T}_{\underline{u}, \nu}$ and by \eqref{X-e}
\begin{equation}\label{map-rel-Xi}
X_i \cdot \e_i\mathcal{T}_{\underline{u}, \nu}=(\e_i-1)X_i  \mathcal{T}_{\underline{u}, \nu}=\e_i X_i \mathcal{T}_{\underline{u}, \nu}-X_i  \mathcal{T}_{\underline{u}, \nu}.
\end{equation} 
As $\e_i \mathcal{F}_{\underline{u}-|m_j|, \nu} \subseteq \mathcal{F}_{\underline{u}-|m_j|, \nu+1}$ so we have $\e_i \mathcal{T}_{\underline{u}, \nu} \subseteq \mathcal{T}_{\underline{u}, \nu+1}$ for each $\underline{u} \in \Z^d$. In particular, $\e_i X_i \mathcal{T}_{\underline{u}, \nu} \subseteq \mathcal{T}_{\underline{u}+e_i, \nu+1}$, whereas $X_i \mathcal{T}_{\underline{u}, \nu} \subseteq \mathcal{T}_{\underline{u}+e_i, \nu}$. The claim follows from \eqref{map-rel-Xi}.
\end{proof}

Using \eqref{partial-e} and similar arguments as 
above, one can show that the map 
\[\gr_{\mathcal{T}_{\underline{u},\star}}(M_{\underline{u}}) \xrightarrow{\cdot \partial_i} \gr_{\mathcal{T}_{\underline{u}-e_i,\star}}(M_{\underline{u}-e_i})(e_i)\]
induced from the $A$-linear map $M_{\underline{u}} \overset{\cdot \partial_i}{\lrt} M_{\underline{u}-e_i}$ is $\gr_{\FF_{\underline{0},\star}}(\D_{\underline{0}})$-linear for $i=1, \ldots, d$.

\vspace{0.15cm}
For any $\underline{a} \in \Z^d$, set $x_{\underline{a}}=\mathfrak{b}_1 \cdots \mathfrak{b}_d$, where 
\begin{equation*}
\mathfrak{b}_i = \begin{cases}
X_i^{a_i} & \text{if } a_i>0,\\
\partial_i^{-a_i} & \text{if } a_i< 0, \\
1 & \text{if } a_i=0
\end{cases}
\end{equation*}
We put $x_{\underline{a}}^+=\prod_{a_i>0} \mathfrak{b_i}$, and  $x_{\underline{a}}^-=\prod_{a_i<0} \mathfrak{b_i}$. Note that $x_{\underline{a}}$'s form a monomial basis of $\D$. 

\begin{lemma}\label{D_0_a_rel}
	For any $\underline{a} \in \Z^d$,
	\begin{equation}\label{rel-D0}
	\D_{\underline{a}}=x_{\underline{a}}^+ \D_{\underline{0}} x_{\underline{a}}^- \quad \mbox{and} \quad  \D_{\underline{a}}=x_{\underline{a}} \D_{\underline{0}}= \D_{\underline{0}} x_{\underline{a}}
	\end{equation}
\end{lemma}

\begin{proof}
	Note that $\D_{\underline{u}}=\{c_{\underline{a}, \underline{b}} X^{\underline{a}} \partial^{\underline{b}} \mid c_{\underline{a}, \underline{b}} \in \Lambda, ~a_i - b_i=u_i \mbox{ for al } i=1, \ldots, r\}$. Let $U, V$ be two nonempty subsets of $\mathcal{S}$ such that $U \cup V=\mathcal{S}$ with $U \cap V =\phi$ and \[u_i=
	\begin{cases}
	v_i & \mbox{ if } i \in U\\
	-v_i & \mbox{ if } i \in V
	\end{cases}\]
	for some $\underline{v} \in \N^d$. Therefore, $a_i=v_i+b_i$ if $i \in U$ and $b_j=a_j+v_j$ if $j \in V$. Note that for any $i \in U$ and $j \in V$,	
	\begin{align*}
	&X^{\underline{a}} \partial^{\underline{b}}=X_1^{a_1} \dots X_{i-1}^{a_{i-1}}X_i^{v_i+b_i}X_{i+1}^{a_{i+1}} \cdots X_d^{a_d}\partial^{\underline{b}}=X_i^{v_i}\left(X_1^{a_1} \dots X_{i-1}^{a_{i-1}}X_i^{b_i}X_{i+1}^{a_{i+1}} \cdots X_d^{a_d}\partial^{\underline{b}}\right),\\
	\mbox{and } &X^{\underline{a}} \partial^{\underline{b}}=X^{\underline{a}}\partial_1^{b_1} \dots \partial_{j-1}^{b_{j-1}}\partial_j^{a_j+v_j}\partial_{j+1}^{b_{j+1}} \cdots \partial_d^{b_d}=\left(X^{\underline{a}}\partial_1^{b_1} \dots \partial_{j-1}^{b_{j-1}}\partial_j^{a_j}\partial_{j+1}^{b_{j+1}} \cdots \partial_d^{b_d}\right)\partial_j^{v_j}.
	\end{align*}
	Repeating this we get that 
	\[X^{\underline{a}} \partial^{\underline{b}}=\prod_{i \in U} X_i^{v_i} \prod_{i \in U} X_i^{b_i} \prod_{j \in V} X_j^{a_j}  \prod_{j \in V} \partial_j^{a_j} \prod_{i \in U} \partial_i^{b_i} \prod_{j \in V} \partial_j^{v_j} \in x_{\underline{u}}^+ \D_{\underline{0}} x_{\underline{u}}^-,\]
	as $\prod_{i \in U} X_i^{v_i}= x_{\underline{u}}^+$ and $\prod_{j \in V} \partial_j^{v_j}=x_{\underline{u}}^-$. Notice $\e_i$ commutes with $X_j, \partial_j$ for each $j \neq i$.
	
	\vspace{0.15cm}
	Take $i \in \{1, \ldots, d\}$ and fix it. Recall that $\partial_iX_i=1+X_i \partial_i$. Now
	\begin{align}\label{X-e}
     \begin{split}
	X_i^j \e_i=&X_i^{j}(X_i \partial_i)\\
	=&X_i^j(\partial_iX_i-1)\\
	=&X_i^{j-1}(X_i \partial_i)X_i-X_i^j\\
	=&\left(X_i^{j-2}(X_i\partial_i)X_i-X_i^{j-1}\right)X_i-X_i^j\\
	=&X_i^{j-2}(X_i\partial_i)X_i^2-2X_i^j\\
	&\vdots\\
	=&(X_i\partial_i)X_i^j-jX_i^j\\
	=&(\e_i-j)X_i^j.
	\end{split}
	\end{align} 
	From the above we also get that $\e_i X_i^j=X_i^j\e_i+X_i^j=X_i^j(\e_i+j)$. Hence 
 $\e_i^wX_i^u=\e_i^{w-1}(\e_i X_i^u)= \e_i^{w-1}X_i^u(\e_i+u)=\e_i^{w-2}X_i^u(\e_i+u)^2= \cdots= X_i^u (\e_i+u)^w$ for $u, v \geq 1$. Thus  \begin{align}\label{e-X-multi}
 \begin{split}
 \e^{\underline{w}} X_i^u&=\e_1^{w_1} \cdots \e_{i-1}^{w_{i-1}} \left(\e_i^{w_i}X_i^u\right)\e_{i+1}^{w_{i+1}}\cdots \e_d^{w_d}\\
 &=X_i^u\e_1^{w_1} \cdots \e_{i-1}^{w_{i-1}} (\e_i+u)^{w_i}\e_{i+1}^{w_{i+1}}\cdots \e_d^{w_d}
  \end{split}
  \end{align}
for each $\underline{w} \in \Z^d$ and $u \geq 1$. Again 
	\begin{align}\label{partial-e}
	\begin{split}
	\e_i\partial_i^j=&(X_i \partial_i)\partial_i^j\\
	=&(\partial_iX_i-1)\partial_i^j\\
	=&\partial_i(X_i \partial_i)\partial_i^{j-1}-\partial_i^j\\
	=&\partial_i\left(\partial_i(X_i\partial_i)\partial_i^{j-2}-\partial_i^{j-1}\right)-\partial_i^j\\
	=&\partial_i^2(X_i\partial_i)\partial_i^{j-2}-2\partial_i^j\\
	&\vdots\\
	=&\partial_i^j(X_i\partial_i)-j\partial_i^j\\
	=&\partial_i^j(\e_i-j).
	\end{split}
	\end{align} 
	Therefore, $\partial_i^j\e_i =\e_i\partial_i^j+X_i^j=(\e_i+j)\partial_i^j$. Hence
	 \begin{align}\label{e-partial-multi}
	\begin{split}
	\e^{\underline{w}} \partial_i^u&=\e_1^{w_1} \cdots \e_{i-1}^{w_{i-1}} \left(\e_i^{w_i}\partial_i^u\right)\e_{i+1}^{w_{i+1}}\cdots \e_d^{w_d}\\
	&=\partial_i^u\e_1^{w_1} \cdots \e_{i-1}^{w_{i-1}} (\e_i-u)^{w_i}\e_{i+1}^{w_{i+1}}\cdots \e_d^{w_d}
	\end{split}
	\end{align}
for each $\underline{w} \in \Z^d$ and $u \geq 1$. Fix $\underline{u} \in \Z^d$. Take $b \in x_{\underline{u}}^+ \D_{\underline{0}} x_{\underline{u}}^-$. Then $b=x_{\underline{u}}^+ f(\e_1, \ldots, \e_d) x_{\underline{u}}^-$ for some $f \in \D_{\underline{0}}=\Lambda[\e_1, \ldots, \e_d]$. In light of \eqref{e-X-multi}, $b=f(\e'_1, \ldots, \e'_d) x_{\underline{a}} \in \D_{\underline{0}} x_{\underline{u}}$, where

	\[\e'_i=
	\begin{cases}
	\e_i-u_i & \mbox{ if } i \in U\\
	\e_i & \mbox{ if } i \in V.
	\end{cases}\]
	Clearly, $\D_{\underline{u}} \subseteq x_{\underline{u}}^+ \D_{\underline{0}} x_{\underline{u}}^- \subseteq \D_{\underline{0}} x_{\underline{u}} \subseteq \D_{\underline{0}} \D_{\underline{u}} \subseteq \D_{\underline{u}}$ and hence $\D_{\underline{u}} = x_{\underline{u}}^+ \D_{\underline{0}} x_{\underline{u}}^- = \D_{\underline{0}} x_{\underline{u}}$. 
	
	Similarly, for any
	$b'\in x_{\underline{u}}^+ \D_{\underline{0}} x_{\underline{u}}^-$, let $b'=x_{\underline{u}}^+ g(\e_1, \ldots, \e_d) x_{\underline{u}}^-$ for some $g \in \D_{\underline{0}}$. In view of \eqref{e-partial-multi}, $b'=x_{\underline{u}} g(\e''_1, \ldots, \e''_d) \in x_{\underline{u}} \D_{\underline{0}}$, where
	\[\e''_i=
	\begin{cases}
	\e_i & \mbox{ if } i \in U\\
	\e_i-u_i & \mbox{ if } i \in V.
	\end{cases}\]
	The result follows.
\end{proof}

\begin{remark}
The left and right module structure of $\D_{\underline{u}}$ as $\D_{\underline{0}}$-modules are different. In particular, take $\underline{u}=e_1 \in \Z^d$. If $\e_1\e_2 X_1=X_1\e_1\e_2$ then $(\e_1+1) \e_2=\e_1\e_2$, as $\D_{\underline{0}}[X_1, \ldots, X_d] \subset \D$ is a domain. Thus $\e_2=0$, a contradiction. 
\end{remark}

Recall that a holonomic $A_{\underline{1}}(K)$-module is cyclic, see \cite[1.8.19]{Bjo}. We prove the following. 
\begin{lemma}\label{cyclic}
Let $L$ be a cyclic $\D$-module. Then $L_{\underline{u}}$ is a cyclic $\D_{\underline{0}}$-module. 
\end{lemma}

\begin{proof}
	Let $L =\D y$ for some homogeneous element $y$ in $L$ with $\deg y=\underline{a}$. In view of equation \eqref{rel-D0} it follows that 
	\[L_{\underline{u}}=\D_{\underline{u}-\underline{a}}y= \D_{\underline{0}} x_{\underline{u}-\underline{a}}y\]
	for all $\underline{u} \in \Z^d$.
\end{proof}

\begin{theorem}\label{simple}
Let $M$ be a simple $\D$-module. Fix $\underline{u}, \underline{v} \in \Z^d$. If both $M_{\underline{u}}$ and $M_{\underline{v}}$ are nonzero, then
	\[\ee(M_{\underline{u}})=\ee(M_{\underline{v}}) \quad \mbox{ and } \quad \mathrm{d}_M(\underline{u})=\mathrm{d}_M(\underline{v}).\]
\end{theorem}
\begin{proof}
	Let $\underline{u}, \underline{v} \in \Z^d$. Pick two nonzero elements $y \in M_{\underline{u}}$ and $z \in M_{\underline{v}}$. Since $M$ is a simple $\D$-module, 
	\[M=\D y=\D z.\] 
	Thus $M_{\underline{u}}=\D_{\underline{0}} y$, and $M_{\underline{v}}= \D_{\underline{0}} z$. Let $z= \beta y$ and $y=\alpha z$ for some $\alpha, \beta \in \D$. Clearly, $y=\alpha \beta y$. Note that $\deg \beta=\deg z-\deg y = \underline{v}- \underline{u}$, and $\deg \alpha=\deg y-\deg z = \underline{u}- \underline{v}$. As $\alpha \in \D_{\underline{u}-\underline{v}}= \D_{\underline{0}} x_{\underline{u}-\underline{v}}$ and $ \beta \in \D_{\underline{0}} x_{\underline{v}-\underline{u}}$ so we have $\alpha= s_0x_{\underline{u}-\underline{v}}$ and $\beta= s'_0x_{\underline{v}-\underline{u}}$ for some $s_0, s'_0 \in \D_{\underline{0}}$. Note that for each $\nu \geq 0$,
	\[\mathcal{T}_{\underline{u}, \nu}=\mathcal{F}_{\underline{0}, \nu} y \quad \mbox{ and } \quad \mathcal{T}_{\underline{v}, \nu}=\mathcal{F}_{\underline{0}, \nu} z.\]
Define 
	the map $\mu_{\beta}: M_{\underline{u}} \to M_{\underline{v}}$ by $\xi \mapsto \beta \xi$ and the map $\mu_{\alpha}: M_{\underline{v}}\rightarrow M_{\underline{u}}$ by $\zeta \mapsto \alpha \zeta$. Then for any $\theta \in \D_{\underline{0}}$, 	
	\begin{align*}
	\mu_{\beta}(\theta \xi) =\beta \theta \xi=&s'_0x_{\underline{v}-\underline{u}} \theta \xi\\
	=& s'_0 \theta x_{\underline{v}-\underline{u}} \xi \quad \mbox{by }\eqref{rel-D0}\\
	=& \theta s'_0 x_{\underline{v}-\underline{u}} \xi \quad \mbox{as } \D_{\underline{0}} \mbox{ is a commutative ring and } s'_0, \theta \in \D_{\underline{0}}\\
	=& \theta \mu_{\beta}(\xi).
	\end{align*}
	Hence the map $\mu_{\beta}$ is $\D_{\underline{0}}$-linear. Similarly, one can show that the map $\mu_{\alpha}$ is $\D_{\underline{0}}$-linear. The two maps $\mu_{\alpha}$ and $\mu_{\beta}$ induce the following maps
	\[\mathcal{T}_{\underline{0}, \nu} y \overset{\cdot \beta}{\lrt} \mathcal{T}_{\underline{0}, \nu} z \quad \mbox{ and } \quad \mathcal{T}_{\underline{0}, \nu} z \overset{\cdot \alpha}{\lrt} \mathcal{T}_{\underline{0}, \nu} y\]
for every $\nu \geq 0$. Thus we get the induced commutative diagram of $\gr_{{}^I\FF_{\underline{0},\star}}(\D_{\underline{0}})$-modules
	\[\xymatrix{\gr_{\mathcal{T}_{\underline{u}, \star}}(M_{\underline{u}}) \ar[r]^{\overline{\mu_{\beta}}} \ar[rd]_{\overline{\Id_{M}}} &\gr_{\mathcal{T}_{\underline{v}, \star}}(M_{\underline{v}})\ar[d]^{\overline{\mu_{\alpha}}}\\
		&\gr_{\mathcal{T}_{\underline{u}, \star}}(M_{\underline{u}}).}\]
Clearly, $\overline{\mu_{\alpha}} \circ \overline{\mu_{\beta}}=\overline{\Id_{M}}$. Hence $\gr_{\mathcal{T}_{\underline{u}, \star}}(M_{\underline{u}}) \cong \gr_{\mathcal{T}_{\underline{v}, \star}}(M_{\underline{v}})$ for every pair $\underline{u}, \underline{v} \in \Z^d$. 
The result follows.
\end{proof}

\begin{proposition}\label{ses-dim-mul}
Let $0 \to M_1 \xrightarrow{f_1} M_2 \xrightarrow{f_2} M_3 \to 0$ be a short exact sequence of finitely generated $\Z^d$-graded $\D$-modules. Then $\mathrm{d}_{M_2}(\underline{u})=\max\{\mathrm{d}_{M_1}(\underline{u}), \mathrm{d}_{M_3}(\underline{u})\}$ for all $\underline{u}\in \Z^d$. 

Moreover, 
\[\mathrm{e}\left((M_2)_{\underline{u}}\right)=
\begin{cases}
\mathrm{e}\left((M_3)_{\underline{u}}\right) & \mbox{ if } \mathrm{d}_{M_1}(\underline{u})< \mathrm{d}_{M_2}(\underline{u}),\\
\mathrm{e}\left((M_1)_{\underline{u}}\right) & \mbox{ if } \mathrm{d}_{M_3}(\underline{u})<\mathrm{d}_{M_2}(\underline{u}).
\end{cases}\]
\end{proposition}

\begin{proof}
We may assume that $f_1$ is an inclusion and $f_2$ is the quotient map. Let $\mathcal{T}^2$ be a $\Z^d$-graded $\FF$-compatible good filtration on $M_2$ (for instance, one can take $\mathcal{T}^2$ as in \eqref{T}). Then $\mathcal{T}^1=\{\mathcal{T}^2_\nu \cap M_1\}_{\nu \in \Z}$ and $\mathcal{T}^3=\{\frac{\mathcal{T}^2_\nu+M_1}{M_1}\}_{\nu \in \Z}$ are $\Z^d$-graded $\FF$-compatible filtrations on $M_1$ and $M_3 \cong M_2/M_1$, respectively. As $\frac{\mathcal{T}^2_\nu+M_1}{M_1} \cong \frac{\mathcal{T}^2_\nu}{\mathcal{T}^2_\nu\cap M_1}$ so we get a short exact sequence 
\begin{equation}\label{ses_gr}
0 \to \gr_{\mathcal{T}^1}(M_1) \to \gr_{\mathcal{T}^2}(M_2) \to \gr_{\mathcal{T}^3}(M_3)\to 0
\end{equation}
of $\Z^{d}$-graded $\gr_{\mathcal{F}}(\D)$-modules. Since $\gr_{\mathcal{T}^2}(M_2)$ is a finitely generated module over the Noetherian ring $\gr_{\mathcal{F}}(\D)$, it follows that $\gr_{\mathcal{T}^1}(M_1)$ and $\gr_{\mathcal{T}^3}(M_3)$ are finitely generated $\gr_{\mathcal{F}}(\D)$-modules. Besides, \eqref{ses_gr} induces a short exact sequence 
\[0 \to \gr_{\mathcal{T}^1_{\underline{u},*}}\left((M_1)_{\underline{u}}\right) \to \gr_{\mathcal{T}^2_{\underline{u},*}}\left((M_2)_{\underline{u}}\right) \to \gr_{\mathcal{T}^3_{\underline{u},*}}\left((M_3)_{\underline{u}}\right) \to 0\]
of finitely generated graded $\gr_{\mathcal{F}_{\underline{0},*}}(\D_{\underline{0}})$ for every $\underline{u} \in \Z^d$. The statements follow from the definitions (see \ref{dim-mul}). 
\end{proof}

We now relate the Bernstein-type dimension and multiplicity of each component of a finite length $A_{\underline{1}}(\Lambda)$-module to the respective invariants of that component of certain simple $A_{\underline{1}}(\Lambda)$-modules.
\begin{lemma}\label{finLength-dim-mul}
Suppose that $M$ is a finite length $A_{\underline{1}}(\Lambda)$-module. Let $0 \subsetneq N_0 \subsetneq N_1 \subsetneq \cdots \subsetneq N_t=M$ be a filtration of $\Z^d$-graded submodules of $M$ such that $N_i/N_{i-1} \cong L_i$ is simple $\Z^d$-graded. Then for every $\underline{u} \in \Z^d$,	
\[d_M(\underline{u})=\max\{d_{N_i}(\underline{u})\mid i=0, \ldots, t\} \quad \mbox{ and } \quad \ee(M_{\underline{u}})=\sum_{d_{N_j}(\underline{u})=d_M(\underline{u})}\ee\left((N_j)_{\underline{u}}\right).\]
\end{lemma}

\begin{note*}
The length `$t$' and the set $\mathcal{J}(M)=\{L_1, \ldots, L_t\}$ are unique by the \emph{Jordan-Holder Theorem}.  
Moreover, for each $\underline{u} \in \Z^d$,
\[0 \subsetneq \left(N_0\right)_{\underline{u}} \subset \left(N_1\right)_{\underline{u}}  \subset \cdots \subset \left(N_t\right)_{\underline{u}} =M_{\underline{u}}\]
is a filtration of $M_{\underline{u}}$ as a $\D_{\underline{0}}$-module. By Lemma \ref{cyclic}, nonzero $\left(N_i/N_{i-1}\right)_{\underline{u}}$'s are cyclic.
\end{note*}

\begin{proof}
The result follows from Proposition \ref{ses-dim-mul}.
\end{proof}

From the above lemma, we right away get a primary result of this section.

\begin{theorem}\label{multi-Bdim-mul}
Let $M$ be a $\Z^d$-graded holonomic $\D$-module. Fix a subset $U$ of $\mathcal{S}$. Suppose that $M_{\underline{u}}$ is nonzero for some $\underline{u} \in \mathcal{B}\left(\underline{a}^U\right)$. Then for all $\underline{u} \in \mathcal{B}\left(\underline{a}^U\right)$,
\[\ee(M_{\underline{u}})=\ee(M_{\underline{a}^{U}}) \quad \mbox{ and } \quad \mathrm{d}_M(\underline{u})=\mathrm{d}_M(\underline{a}^{U}).\]
\end{theorem}

\begin{proof}
By Lemma \ref{fl_hol}, $M$ has finite length. Therefore, the statement follows from Theorem \ref{simple} and Lemma \ref{finLength-dim-mul}. 
\end{proof}

\section{A structure theorem of the components}

Suppose that $A$ is a Dedekind domain of characteristic zero such that its localization at every maximal ideal has mixed characteristic with finite residue field. 
Under this framework, in \cite{PutRoy22_pre}, the authors presented a structure theorem for the components of $H^i_I(R)$ for all $i \geq 0$. They further showed that if $A$ is a PID then each component can be written as a direct sum of its torsion part and torsion-free part. We now prove an analogue of the structure theorem in characteristic zero where $A=K[[Y]]$. Observe that unlike the mixed characteristic case, the torsion part does not have any finitely generated summand.

\begin{theorem}\label{struc}
Let $K$ be a field of characteristic zero, $A=K[[Y]]$ be a power series ring in one variable and let $Q(A)$ denote the field of fractions of $A$. Let $R=A[X_1, \ldots, X_d]$ be a standard $\N^d$-graded polynomial ring over $A$.
Let $I \subseteq R$ be a $\mathfrak{C}$-monomial ideal. Fix $\underline{u} \in \Z^d$. Then 
\[H^i_I(R)_{\underline{u}}\cong E(A/(Y))^{s({\underline{u}})} \oplus Q(A)^{v({\underline{u}})} \oplus A^{r({\underline{u}})}\]
for finite numbers $s({\underline{u}}), v({\underline{u}}), r({\underline{u}})$.
\end{theorem}
\begin{proof}
We set $\m=(Y)$ and put $N=H^i_I(R)_{\underline{u}}$. Observe that \[\Gamma_{\m}(N)\cong \Gamma_{\m R}\left(H^i_I(R)\right)_{\underline{u}}.\]
Since $\Gamma_{\m R}\left(H^i_I(R)\right)$ is a $\Z^d$-graded generalized Eulerian $A_{\underline{1}}(A)$-module, by Theorem \ref{injdim-dim},
\[\injdim \Gamma_{\m}(N) \leq \dim \supp \Gamma_{\m}(N)=0.\]
Hence $\Gamma_{\m}(N)$ is an injective module. So
\[\Gamma_{\m}(N) \cong E_A(A/(Y))^{s(\underline{u})}.\] 
Thus by Theorem \ref{injdim-dim}, $s(\underline{u})=\mu_0(\m, N)$ is finite. 

Clearly, $t(N)=\Gamma_{\m}(N)$ and the torsion-free part of $N$ is $\overline{N}:=\frac{N}{\Gamma_\m(N)}$. As $\Gamma_\m\left(\overline{N}\right)=0$ so $Y$ is $\overline{N}$-regular. Consider the short exact sequence
\[0 \to \Gamma_\m(N) \to N \to \overline{N} \to 0.\]
Since $\Gamma_\m(N)$ is an injective module, the above sequence splits. Thus $N \cong \Gamma_\m(N) \oplus \overline{N}$. Consider the $A$-submodule $V:=\bigcap_{n=1}^\infty Y^n \overline{N} \subseteq \overline{N}$. 

By $\Lambda$ we denote the ring of differential operators of $A$, that is, $\Lambda \cong A \langle \partial/\partial Y \rangle$. Since $M$ and $\Gamma_{\m R}(M)$ are $\Z^d$-graded $A_{\underline{1}}(\Lambda)$-modules, $N$ and $\Gamma_\m(N)$ are $\Lambda$-modules. Thus $\overline{N}$ is a $\Lambda$-module. We remark that $V$ is a $\Lambda$-module. It is enough to show $\frac{\partial}{\partial Y}(V) \subseteq V$. Take $v \in V$. Then $v \in Y^{n+1} \overline{N}$ for all $n\geq 1$. Fix $n$ and put $v=Y^{n+1}w$ for some $w \in \overline{N}$. Notice 
\[\frac{\partial}{\partial Y}(v)=\frac{\partial}{\partial Y} \left(Y^{n+1}w\right)=Y^{n+1} \cdot \frac{\partial}{\partial Y}(w)+(n+1)w Y^n \in Y^n \overline{N}.\]
The statement follows.

\begin{claim} \label{V-N}
Take $v \in V$.
Let $v=Y^a w$ for some $w \in \overline{N}$ and $a \geq 1$. Then $w \in V$.
\end{claim}

Recall $V \subseteq Y^n \overline{N}$ for all $n \geq 1$. Fix $n$. 
Let $v=Y^{n+a}w_{n+a}$ for $w_{n+a} \in \overline{N}$, that is, $Y^aw=Y^{n+a}w_{n+a}$. As $Y$ is $\overline{N}$-regular so is $Y^a$. It follows that $w=Y^n w_{n+a} \in Y^n \overline{N}$ for each $n \geq 1$. Hence $w \in V$.

\vspace{0.2cm}
Notice that any $A$-regular element is of the form $Y^a$ for some $a \geq 1$. For any $v \in V$ we have $v=Y^a w$ for some $w \in \overline{N}$ and form Claim \ref{V-N} it follows that $w \in V$. Hence the $A$-module $V$ is divisible. Since $A$ is a PID, by \cite[Corollary 3.1.5]{BruHer}, $V$ is injective. Now $Y$ is $V$-regular, as it is $\overline{N}$-regular. So $V \cong E_{A_{(0)}}(A_{(0)})^{v(\underline{u})}=Q(A)^{v(\underline{u})}$.
%
%
%
%
%
%
Observe that $\overline{N} \otimes_A Q(A) \cong N \otimes_A Q(A)$. Set $W:=Q(A)[X_1, \ldots, X_d]$. Then $R \subseteq W$ is a subring and $M \otimes_R W \cong M \otimes_A Q(A) \cong H^i_{IW}(W)$. Since $IW \subseteq W$ is a monomial ideal, $\dim_K H^i_{IW}(W)_{\underline{u}}$ is finite. Moreover, $V \otimes_{Q(A)} Q(A) \subseteq \overline{N} \otimes_{Q(A)} Q(A)$. So $v(\underline{u})$ is finite.

We put $M :=H^i_I(R)$ and set $\overline{M} :=M/\Gamma_{\m R}(M)$. Clearly, $\overline{M}_{\underline{u}}=\overline{N}$. Besides, $\overline{M}$ is $\Z^d$-graded generalized Eulerian, holonomic $A_{\underline{1}}(A)$-module. From Proposition \ref{Koszul-local} we have $H_0(Y, \overline{M})=\overline{M}/Y \overline{M}$ is a $\Z^d$-graded generalized Eulerian, holonomic $A_{\underline{1}}(K)$-module. So by Theorem \ref{hol-mon}, $\dim_k \left(\overline{M}/Y \overline{M}\right)_{\underline{u}}< \infty$. 
We set
$V:=\bigcap_{j=1}^\infty Y^j \overline{N}$.
Notice that $\dim_k (\overline{N}/V)/Y( \overline{N}/V)< \infty$, since we have $(\overline{N}/V)/Y( \overline{N}/V) \cong \overline{N}/Y \overline{N}=\left(\overline{M}/Y \overline{M}\right)_{\underline{u}}$. Further, $\bigcap_{j=1}^\infty Y^j \left(\overline{N}/V\right)=\left(\bigcap_{j=1}^\infty Y^j\overline{N}\right)/V=0$.
Thus by \cite[Theorem 8.4]{Mat}, it follows that $\overline{N}/V$ is finitely generated. Besides, 
the map $\overline{N}/V \xrightarrow{ \cdot Y} \overline{N}/V$ is injective by Claim \ref{V-N}. Hence the $A$-module $\overline{N}/V$ is torsion-free. Therefore, from the structure theorem for finitely generated modules over the PID $A=K[[Y]]$ we get
\[\overline{N}/V\cong A^{r(\underline{u})}\]
for some finite number $r(\underline{u}) \geq 0$. From the short exact sequence
\begin{equation*}
0 \to V \to \overline{N} \to \overline{N}/V \to 0
\end{equation*}
it follows that
\[\overline{N} \cong V \oplus \overline{N}/V \cong Q(A)^{v(\underline{u})} \oplus A^{r(\underline{u})}.\]
\end{proof}

\begin{remarks}

\noindent	
\begin{enumerate}
\item Since $Q(A)$ is a flat $A$-module, $\overline{N}$ is a flat $A$-module.

\item 
Take a nonempty subset $U$ of $\mathcal{S}$. As $\Gamma_{\m R}\left(H^i_I(R)\right)$ is a $\Z^d$-graded generalized Eulerian $A_{\underline{1}}(A)$-module so from
Theorem \ref{multi-bass} we get $s(\underline{u})=s(\underline{a}_{U})$ for all $\underline{u} \in U$. Furthermore, $\overline{M}$ is a $\Z^d$-graded generalized Eulerian $A_{\underline{1}}(A)$-module and $\overline{N}=\overline{M}_{\underline{u}}$. Hence by \eqref{comp_rel},
\[Q(A)^{v(\underline{u})} \oplus A^{r(\underline{u})} \cong Q(A)^{v(\underline{a}_{U})} \oplus A^{r(\underline{a}_{U})}\] 
for each $\underline{u} \in U$. So rank of these $A$-modules are the same. Therefore, \[v(\underline{u})+r(\underline{u})=v(\underline{a}_{U})+r(\underline{a}_{U}).\] 
By \eqref{comp_rel}, there is an $A$-module isomorphism $\overline{M}_{\underline{u}} \cong \overline{M}_{\underline{a}_U}$ via the map `$f$' (say). 
So we get an induced isomorphism 
\[Y^a\overline{M}_{\underline{u}} \cong f \left(Y^a\overline{M}_{\underline{u}}\right)=Y^af \left(\overline{M}_{\underline{u}}\right)=t^a\overline{M}_{\underline{a}_U}\] for all $\underline{u} \in U$ and each $a \geq 1$. Therefore, $\bigcap_{j=1}^\infty Y^j \overline{M}_{\underline{u}} \cong \bigcap_{j=1}^\infty Y^j \overline{M}_{\underline{a}_U}$ as $A$-modules. It thus follows that $v(\underline{u})=v(\underline{a}_{U})$ and consequently $r(\underline{u})=r(\underline{a}_{U})$. 
\end{enumerate} 
\end{remarks}

\section{Examples}

In this section, we consider $A:=K[[Y]]$. We first discuss an example where $s(\underline{u}) \neq 0$ in Theorem \ref{struc} for some $\underline{u} \in \Z^d$.

\begin{example}\label{eg1}
Let $R=A[X]$. Take the $\mathfrak{C}$-monomial ideal $I=(YX)$ in $R$. Note that $(YX)=(Y) \cap (X)$. 
Consider the Mayer-Vietoris sequence
\[\cdots \to H^1_{(Y, X)}(R) \to H^1_{(Y)}(R) \oplus H^1_{(X)}(R) \to H^1_{(YX)}(R) \to H^2_{(Y, X)}(R) \to 0.\]
As $\hgt (Y, X)=2$ so we have $H^1_{(Y, X)}(R)=0$. Moreover,
\[H^1_{(Y)}(R) \cong H^1_{(Y)}(A)[X]\cong E_A(A/(Y))[X] \quad \mbox{ and } \quad H^1_{(X)}(R) \cong A[X^{-1}](-1).\]
Therefore,
\[H^1_{(YX)}(R)_u \supseteq
\begin{cases}
E_A(A/(Y)) & \mbox{ if } u \geq 0\\
A  & \mbox{ otherwise }.
\end{cases}
\]
So $E_A(A/(Y))$ is a direct summand of $H^1_{(YX)}(R)_u$ and hence of $\Gamma_{(Y)}\left(H^1_{(YX)}(R)_u\right)$.
\end{example}

We now give an example where for some $\underline{u} \in \Z^d$ the torsion-free part of $\overline{M_{\underline{u}}}$ is not finitely generated, equivalently, $v(\underline{u}) \neq 0$ in Theorem \ref{struc}.

\begin{example}
Suppose that $R=A[X_1, X_2]$. Take the $\mathfrak{C}$-monomial ideal $I=(YX_1, X_2)$ in $R$. We set $\overline{R}:=\frac{A}{(Y)}[X_1, X_2]\cong K[X_1,X_2]$. Consider the short exact sequence
\[0 \to R \xrightarrow{\cdot Y} R \to\overline{R} \to 0,\]
which induces a long exact sequence
\[\cdots \to H^1_{(Y)}(\overline{R}) \to H^2_{(YX_1, X_2)}(R) \xrightarrow{\cdot Y} H^2_{(YX_1, X_2)}(R) \to H^2_{(Y)}(\overline{R}) \to \cdots.\]
Observe that $H^2_{(Y)}(\overline{R})=0$ as $\hgt (Y)=1$ and $\dim_K H^1_{(Y)}(\overline{R})_{(u_1, u_2)}$ is finite for each pair ${(u_1, u_2)} \in \Z^2$. We put $M:= H^2_{(YX_1, X_2)}(R)$. Then we have $M_{(u_1, u_2)} \xrightarrow{\cdot Y} M_{(u_1, u_2)} \to 0$ for all ${(u_1, u_2)} \in \Z^2$.  
Moreover, from the commutative diagram
\[
	\xymatrix@C=0.75em@R=0.75em{
		M_{(u_1, u_2)} \ar[r] \ar[d]^{\cdot Y} & \overline{M_{(u_1, u_2)}} \ar[d]^{\cdot Y} \ar[r] & 0\\
		M_{(u_1, u_2)} \ar[r] \ar[d]& \overline{M_{(u_1, u_2)}} \ar[r] & 0\\
		0 & &
	}\]
and using the snake lemma, we get that
$\overline{M_{(u_1, u_2)}} \xrightarrow{\cdot Y} \overline{M_{(u_1, u_2)}} \to 0$ for all ${(u_1, u_2)} \in \Z^2$. By Theorem \ref{struc}, $\overline{M_{(u_1, u_2)}}$ is a torsion-free $A$-module. If it is finitely generated, then it is a free $A$-module. So in that case the surjective map $\overline{M_{(u_1, u_2)}} \xrightarrow{\cdot Y} \overline{M_{(u_1, u_2)}}$ is an isomorphism. This leads to a contradiction whenever $\overline{M_{(u_1, u_2)}}  \neq 0$. 
\end{example}

\end{document}